\documentclass[11pt]{article}
\usepackage[margin=1in]{geometry}  
\usepackage{graphicx}              
\usepackage{amsmath, amssymb, amsthm, epsfig}               
\usepackage{enumerate}
\usepackage{amsfonts}              
\usepackage{amsthm}                
\usepackage{amsthm}
\usepackage{enumerate}
\theoremstyle{plain}
\usepackage{color}
\usepackage{hyperref}
\hypersetup{colorlinks,citecolor=blue}
\usepackage{cite}
\newtheorem{corollary}{Corollary}[section]
\newtheorem{dfn}[corollary]{Definition}

\newtheorem{lemma}[corollary]{Lemma}
\newtheorem{prop}[corollary]{Proposition}
\newtheorem{rem}[corollary]{Remark}
\newtheorem{thm}[corollary]{Theorem}

\newfont{\sBlackboard}{msbm10 scaled 1200}

\newcommand{\mylabel}[1]{\label{#1}
    \ifx\undefined\stillediting
    \else \fbox{$#1$}\fi }
\newcommand{\BE}{\begin{equation}}

\newcommand{\EEQ}{\end{equation}}
\newcommand{\rfb}[1]{\mbox{\rm
        (\ref{#1})}\ifx\undefined\stillediting\else:\fbox{$#1$}\fi}

\newfont{\Blackboard}{msbm10 scaled 1200}

\newfont{\roma}{cmr10 scaled 1200}

\newcommand{\bb}{\begin{equation}}
\newcommand{\bbb}{\end{equation}}

\newcommand{\mm}    {{\hbox{\hskip 0.5pt}}}

\newcommand{\bluff} {{\hbox{\raise 15pt \hbox{\mm}}}}
scaled\magstep2

\makeatletter
\def\section{\@startsection {section}{1}{\z@}{-3.5ex plus -1ex minus
        -.2ex}{2.3ex plus .2ex}{\large\bf}}
\makeatother
\numberwithin{equation}{section}

\begin{document}
\title{On the  fractional Musielak-Sobolev spaces in $\mathbb{R}^d$: Embedding results \& applications}
\author{Anouar Bahrouni\footnote{
bahounianouar@yahoo.fr ; Anouar.Bahrouni@fsm.rnu.tn}\ \, \ Hlel
Missaoui\footnote{hlel.missaoui@fsm.rnu.tn}\ \ and \ Hichem Ounaies
\footnote{
hichem.ounaies@fsm.rnu.tn}\\
Mathematics Department, Faculty of Sciences, University of Monastir,\\ 5019 Monastir, Tunisia}
\maketitle


\begin{abstract}This paper deals with new continuous and compact embedding theorems for
the fractional Musielak-Sobolev spaces in $\mathbb{R}^d$. As an application, using the variational methods, we obtain the existence  of nontrivial weak
solution for the following Schr\"odinger equation
$$
    (-\Delta)_{g_{x,y}}^s u+V(x)g(x,x,u)=b(x)\vert u\vert^{p(x)-2}u,\ \text{for all}\ x\in \mathbb{R}^d,$$
where  $(-\Delta)_{g_{x,y}}^s$ is the fractional Museilak
$g_{x,y}$-Laplacian, $V$ is a potential function, $b\in
L^{\delta^{'}(x)}(\mathbb{R}^d)$, and $p,\delta\in
C\left(\mathbb{R}^d,(1,+\infty)\right)\cap
L^{\infty}(\mathbb{R}^d)$. We would like to mention that the theory
of the fractional Musielak-Sobolev spaces
 is in a developing state and there are few papers in this topic, see \cite{M1,M8,M9}.   Note that, all these latter works dealt with bounded
case and there are no results devoted for the fractional Musielak-Sobolev spaces in $\mathbb{R}^d$. Since the embedding results are crucial in applying variational methods, this work will provide a bridge between the fractional Mueislak-Sobolev theory and PDE's.
\end{abstract}



\section{Introduction}
Fractional Orlicz-Sobolev spaces (see \cite{M28}) and fractional
Sobolev spaces with variable exponent (see \cite{M33}) are two
distinct extensions of classical fractional Sobolev spaces (see
\cite{M26}), and they are two special kinds of fractional
Musielak-Sobolev spaces (see \cite{M8,M9}). The importance of the
Sobolev-type embedding theorems is well known. We refer the reader
to \cite{M13,Cianchi5} for the embedding of fractional
Orlicz-Sobolev spaces and to \cite{M12,M33} for the embeddings of
fractional Sobolev spaces with variable exponent. The  aim of the
present paper is to establish a continuous and a compact embedding
theorems (Strauss theorem) for the fractional Musielak-Sobolev
spaces $\left(W^{s,G_{x,y}}(\Omega)\right)$ in $\mathbb{R}^d$. This
is a new research topic.\\ Precisely, our main contributions are the
following:
\begin{enumerate}
    \item[$(1)$] We prove a continuous embedding theorem for the fractional Musielak-Sobolev space in the whole space  $\mathbb{R}^d$.
    \item[$(2)$] We prove a
     compact embedding theorem for the space $W^{s,G_{x,y}}(\mathbb{R}^{d})$ with a weight function.
    \item[$(3)$]  We prove a Lions-type lemma to the modular funtion.
    \item[$(4)$] We prove a Strauss compact embedding theorem for the radial fractional Musielak-Sobolev space.
    \item[$(5)$] We obtain the existence of a nontrivial weak solution for a class of this new
    kind of nonlocal problems.
\end{enumerate}
To the best of our knowledge, this is the first work dealing with embedding theorems for the fractional Musielak-Sobolev spaces in the whole space $\mathbb{R}^d$.\\

Recently, the study of nonlinear equations involving the fractional Laplacian $(-\Delta)^s$, $0<s<1$, has gained
tremendous popularity due to their intriguing analytic structure and in view of several applications
in different subjects, such as Optimization, Finance, Anomalous Diffusion, Phase
Transition, Flame propagation, Minimal surface. Note that, the adequate framework for these types of
nonlinear equations ( involving the fractional Laplace operator ) is the well-known fractional Sobolev space.  For the basic
properties of fractional Sobolev spaces and the operator $(-\Delta)^s$ with applications to partial differential equations, we
refer the interested reader to \cite{M26,M62} and the references therein.\\

 As we mentioned at the begging, the fractional Sobolev space has two distinct extensions. The first one is the fractional Sobolev space
  with variable exponent which was firstly introduced in 2017 by Kaufmann et al. in \cite{M33}. After that, some studies on this
  context have been performed by using different approaches, see \cite{M2,M6,M11,M12,M28,M60,M101,M46}. In these last references,
  the authors established a compact embedding theorems and proved some further qualitative properties of the fractional Sobolev space
with variable exponent and the fractional p(x)-Laplace operator. The second extension is the fractional Orlicz-Sobolev spaces. These new
 spaces built a  bridge between the fractional order theory and the Orlicz-Sobolev theory. As far as we know, J. Fern\`andez. Bonder et al. firstly introduced
the fractional Orlicz-Sobolev space
and the new fractional Orlicz g-Laplace operator, see \cite{M28}. After that in 2020, S. Bahrouni, A. M. Salort, A. Cianchi, et al. proved some basic results as
the embedding theorems and the fundamental topological properties which allow us to apply the
variational approaches, see \cite{Cianchi1,Cianchi2,Cianchi3,47,Cianchi4,Cianchi5,M13,M14,M15,M28,M29,ho,M47}.\\

A natural question is to see if there exists a more general functional space that includes the both extensions at the same time. Very recently, Azroul et al. \cite{M8,M9}, gave the answer to the latter question by considering the new fractional Musielak-Sobolev space $W^{s,G_{x,y}}(\Omega)$ which is the natural generalization of fractional Sobolev with variable exponent and fractional Orlicz-Sobolev spaces. Moreover, they defined the new fractional Musielak $g_{x,y}$-Laplace operator $(-\Delta)^s_{g_{x,y}}$, for all $s\in(0,1)$
\begin{align*}
(-\Delta)^s_{g_{x,y}}u(x):&=\text{p.v.}\int_{\mathbb{R}^d}g_{x,y}\left(\frac{\vert u(x)-u(u)\vert}{\vert x-y\vert^{s}}\right)\frac{dy}{\vert x-y\vert^{d+s}},\ \ \text{for all}\ x\in \mathbb{R}^d
 \end{align*}
 where p.v. is a commonly used abbreviation for "in the principle value sense", $\displaystyle{G_{x,y}(t):=\int_0^{\vert t\vert}g_{x,y}(\tau)d\tau}$, and
 $g_{x,y}:\Omega\times\Omega\times\mathbb{R}\longrightarrow \mathbb{R}$ are a Carath\'eodory
functions that satisfy some suitable assumptions which will be
mentioned later in Section 2. In \cite{M9}, the authors established
a continuous and  compact embedding theorems for the fractional
Musielak-Sobolev space into Musielak spaces in the bounded case. In
\cite{M10}, by using the Ekeland’'s principle in combination with
a direct variational approach,
 the authors  proved the existence of weak solutions for a
nonlocal problem driven by the fractional $g_{x,y}$-Laplacian with a
Neumann and Robin boundary condition. Very recently (in $2023$) J.C.
de Albuquerque et al. proved  some abstract results on the
perspective of the fractional Musielak- Sobolev spaces, such as:
uniform convexity, Radon-Riesz property with respect to the modular
function, $(S_+)$-property, Brezis-Lieb type Lemma to the modular
function and monotonicity results. Add to that, they studied the
existence of weak solutions to the
 problem
$$
\left\lbrace
\begin{array}{ll}
 (-\Delta)^s_{g_{x,y}}u=f(x,u), & \text{in}\ \Omega,\\
 u=0,& \text{on}\ \mathbb{R}^d\setminus\Omega,
\end{array}
\right.
$$
 where $d\geq 2$,
 $\Omega\subset\mathbb{R}^d$ is a bounded domain with Lipschitz boundary
 while $f:\Omega\times\mathbb{R}\longrightarrow \mathbb{R}$
is a Carath\'eodory function, see \cite{M1}.\\

To the best of our  Knowledge, the literature on the fractional
Musielak-Sobolev spaces  and their applications is quite few, see
\cite{M1,M8,M9,M10}. Note that, all these latter works dealt with
bounded case and there are no results devoted for the fractional
Musielak-Sobolev spaces in $\mathbb{R}^d$.
 Motivated by the above discussion, our main goal in this paper is to establish a continuous and a compact embedding theorems for the fractional Musielak-Sobolev spaces in $\mathbb{R}^d$.\\

Our main results are summarized in the following theorems:
 \begin{thm}[Continuous embedding]\label{thm1} Let $G_{x,y}$ be a generalized N-function satisfying the assumptions $(g_1)-(g_5)$ and \eqref{bf}. Then,
 \begin{enumerate}
     \item[$(1)$] the embedding $W^{s,G_{x,y}}(\mathbb{R}^d)\hookrightarrow L^{\widehat{G}_x^*}(\mathbb{R}^d)$ is continuous;
     \item[$(2)$] for any generalized N-function $\widehat{A}_x$ satisfying \eqref{bf},
      \begin{equation}\label{cA}
     1<\ell_{\widehat{A}_x}\leq \frac{\widehat{A}^{'}(x,t)t}{\widehat{A}(x,t)}\leq m_{\widehat{A}_x}<+\infty,\ \ \text{for all}\ x\in \Omega\ \text{and all}\ t>0,
     \end{equation}
      \begin{equation}\label{2eq60}        \widehat{A}_x\prec\prec\widehat{G}_{x}^*,
    \end{equation}
    and
\begin{equation}\label{mla1b}
    \lim_{|t|\to 0}\frac{\widehat{A}_x(t)}{\widehat{G}_x(t)}=0,\ \ \text{uniformly in }\ x\in \mathbb{R}^d,
    \end{equation}
     the embedding $W^{s,G_{x,y}}(\mathbb{R}^d)\hookrightarrow L^{\widehat{A}_x}(\mathbb{R}^d)$ is continuous.
 \end{enumerate}
 \end{thm}
 Next, we define the following subspace of $W^{s,G_{x,y}}(\mathbb{R}^{d})$:
 $$\mathbb{E}:=\bigg{\{}u\in W^{s,G_{x,y}}(\mathbb{R}^{d}):\ \int_{\mathbb{R}^{d}}V(x)\widehat{G}_x(u)dx<\infty\bigg{\}},$$
 where $V$ is a potential function satisfying:
 \begin{enumerate}
    \item[$(V_1)$] there exists $V_0 > 0$ such that $V(x)\geq V_0$ for any $x\in \mathbb{R}^d$;
    \item[$(V_2)$]the set $\lbrace x\in \mathbb{R}^d:\ V(x)<L\rbrace$ has finite Lebesgue measure for each $L>0$.
\end{enumerate}
 \begin{thm}[Compact embedding]\label{comp1}
    Assume that $(g_1)-(g_5)$ and $(V_{1})$--$(V_{2})$ hold. Then, the embedding $\mathbb{E}\hookrightarrow L^{\widehat{G}_x}(\mathbb{R}^d)$ is compact.
\end{thm}
As a consequence of the last theorem, we give the following result:
\begin{thm}[Compact embedding]\label{compact}
    Assume that $(g_1)-(g_5)$ and $(V_{1})$--$(V_{2})$ hold. Let $\widehat{A}_x$ be a generalized N-function satisfying \eqref{bf}, \eqref{cA}, \eqref{2eq60},
  and at least one of the following conditions:
    \begin{itemize}
        \item[$(1)$]  The following limit holds
        \begin{equation}\label{mla1}
        \limsup_{\vert t \vert\rightarrow0}\frac{\widehat{A}_x(\vert t\vert)}{\widehat{G}_x(\vert t\vert)}<+\infty,\ \ \text{uniformly in }\ x\in \mathbb{R}^d.\tag{$G_1$}
        \end{equation}
    \item[$(2)$]There exists $a\in(0,1)$ such that
   \begin{equation}\label{mla2}
\widehat{A}_x(\vert t\vert) \leq  \widehat{G}_x(|t|)^a\widehat{G}_x^*(\vert t\vert)^{1-a},\ \ \text{for all}\ \vert t\vert\leq 1\ \ \text{and}\ x\in \mathbb{R}^d.\tag{$G_2$}
\end{equation}
    \end{itemize}
    Then, the space $\mathbb{E}$ is compactly embedded into $L^{\widehat{A}_x}(\mathbb{R}^{d})$.
\end{thm}
Next, we extend the well-known Lion's lemma to the frame of the new
fractional Musielak-Sobolev spaces.
\begin{thm}[Lions' Lemma type result]\label{lions}
    Suppose that $(g_1)-(g_5)$ and \eqref{bf} hold. Let $\widehat{A}_x$ be a generalized N-function satisfying \eqref{bf}, \eqref{cA}, \eqref{2eq60},
  and \eqref{mla1b}.
   Let $\lbrace u_{n}\rbrace_{n\in \mathbb{N}}$ be a bounded sequence in $W^{s,G_{x,y}}(\mathbb{R}^{d})$ in such way that $u_n \rightharpoonup 0$ in $\mathbb{E}$ and
    \begin{equation}\label{lionss}
    \lim_{n\rightarrow+\infty}\left[\sup_{y\in\mathbb{R}^{d}}\int_{B_{r}(y)}\widehat{G}_x(u_{n})\,\mathrm{d}x \right]=0,\ \text{for some}\ r>0.
    \end{equation}
 Then, $u_{n}\rightarrow0$ in $L^{\widehat{A}_x}(\mathbb{R}^{d})$.
\end{thm}
Denote by
$$W^{s,G_{x,y}}_{rad}(\mathbb{R}^d):=\left\lbrace u\in W^{s,G_{x,y}}(\mathbb{R}^d):\ u\ \text{is radially symmetric}\right\rbrace.$$
By $u$ being radially symmetric, we mean a function $u:\mathbb{R}^d\longrightarrow \mathbb{R}$ satisfying $u(x) = u(y)$ for all $\vert x\vert=\vert y\vert$, $x,y\in \mathbb{R}^d.$\\
Now, using the above new Lion's lemma, we are ready to  give the following variant of Strauss theorem.
\begin{thm}[Strauss radial embedding]\label{thms}
  Let $G_{x,y}$ be a generalized N-function and $s\in (0,1)$. Under the assumptions $(g_1)-(g_5)$ and \eqref{bf}.  Let $\widehat{A}_x$ be a generalized  N-function verifying \eqref{bf}, \eqref{cA}, \eqref{2eq60},
    and \eqref{mla1b}. Then, we have the compact embedding $$W^{s,G_{x,y}}_{rad}(\mathbb{R}^d)\hookrightarrow L^{\widehat{A}_x}(\mathbb{R}^d).$$
\end{thm}

As an application of the above abstract results, we consider the following nonlocal problem:
\begin{equation}\label{PV}
    (-\Delta)_{g_{x,y}}^s u+V(x)g(x,x,u)=b(x)\vert u\vert^{p(x)-2}u,\ \text{for all}\ x\in \mathbb{R}^d,\tag{$\mathcal{P}$}
\end{equation}
where  $(-\Delta)_{g_{x,y}}^s$ is the fractional Museilak $g_{x,y}$-Laplacian, $p,\delta\in C\left(\mathbb{R}^d,(1,+\infty)\right)\cap L^{\infty}(\mathbb{R}^d)$  such that $$1<p^-\leq p^+<g^-\leq g^+\leq \delta^- p^-\leq \delta(x) p(x)\leq \delta^+ p^+\leq g_*^-,\ \text{for all}\ x\in \mathbb{R}^d.$$
We recall that $\displaystyle{p^-:=\inf_{x\in \mathbb{R}^d}p(x)}$and $ \displaystyle{p^+:=\sup_{x\in \mathbb{R}^d}p(x)}$.\\
For what concerns the function $b$, we assume the following assumption:
\begin{equation}\label{B}
  b\in L^{\delta^{'}(x)}(\mathbb{R}^d),\ \text{where}\ \delta^{'}(x):=\frac{\delta(x)}{\delta(x)-1},\ \forall\ x\in \mathbb{R}^d.\tag{B}
\end{equation}

 \begin{thm}\label{thmv}
   Under the assumptions \eqref{B}, $(g_1)-(g_5)$, \eqref{bf}, and $(V_1)-(V_2)$, problem \eqref{PV}  has a nontrivial weak solution $u\in \mathbb{E}$.
 \end{thm}

 The paper is organized as follows. In Section 2, we give some definitions and fundamental properties of the generalized N-functions,
  Musielak and fractional Musielak-Sobolev spaces. In Section 3, we prove Theorem \ref{thm1}. In Section 4, we present our compact embedding result
  of the weight fractional Musielak-Sobolev space. In Sections 5 and 6, we establish the Lions-type lemma and the Strauss compact embedding of the fractional
   Musielak-Sobolev space. In Section 7, we prove the existence of a nontrivial weak solution to the problem \eqref{PV}. Finally, in the last Section,
   we give some concluding remarks, perspectives, and open problems.

\section{Preliminaries}
In order to construct a suitable setting for our main results, we
consider the following definitions and assumptions: We suppose that
$d\geq 2$, $\Omega$
 is an open subset in $\mathbb{R}^d$ and $G:\Omega\times\Omega\times\mathbb{R}\longrightarrow \mathbb{R}$ is a Carath\'eodory
function defined by
$$G_{x,y}(t):=G(x,y,t):=\int_0^{\vert t\vert}g_{x,y}(\tau)d\tau:=\int_0^{\vert t\vert}g(x,y,s)(\tau)d\tau,$$
where
$$g(x,y,t):=\left\lbrace\begin{array}{ll}
a(x,y,t)t & \text{if}\ t\neq 0 \\
\ & \ \\
  0   & \text{if}\ t= 0,
\end{array}
\right.$$
with $a:\Omega\times\Omega\times(0,+\infty)\longrightarrow \mathbb{R}_+$  is a function satisfying the following assumptions:
\begin{enumerate}
    \item[$(g_1)$] $\displaystyle{\lim\limits_{t\rightarrow 0}a(x,y,t)t=0}$, and $\displaystyle{\lim\limits_{t\rightarrow +\infty}a(x,y,t)t=+\infty}$, for a.a. $(x,y)\in \Omega\times\Omega$;
    \item[$(g_2)$] $\displaystyle{t\mapsto a(x,y,t)}$ is continuous on $(0,+\infty)$, for all $(x,y)\in \Omega\times\Omega$;
    \item[$(g_3)$] $\displaystyle{t\mapsto a(x,y,t)t}$ is increasing on $(0,+\infty)$, for all $(x,y)\in \Omega\times\Omega$;
    \item[$(g_4)$]there exist $g^-,g^+\in (1,+\infty)$ such that
\begin{equation*}\label{2eqg}
 1<g^-\leq \frac{a(x,y,t)t^2}{G(x,y,t)}\leq g^+<g^-_*:=\frac{dg^-}{d-sg^-},\ \text{for all}\ (x,y)\in \Omega\times\Omega\ \text{and all}\ t>0
\end{equation*}
where $s\in(0,1)$.
\end{enumerate}
We also consider the function $\widehat{G}:\Omega\times\mathbb{R}\longrightarrow \mathbb{R}$ given by
\begin{equation}\label{2eq9}
\widehat{G}_x(t):=\widehat{G}(x,t)=\int_0^{\vert t\vert}\widehat{g}(x,\tau)d\tau,
\end{equation}
where $\widehat{g}(x,t):=\widehat{a}(x,t)t=a(x,x,t)t$, for all $(x,t)\in \Omega\times(0,+\infty).$
\subsection{Generalized N-functions}
In this subsection, we give some definitions and properties for the
generalized N-functions.
\begin{dfn} Let $\Omega$ be an open subset of $\mathbb{R}^d$. A function $G:\Omega\times \Omega\times\mathbb{R}+\longrightarrow \mathbb{R}$ is called a generalized N-function if it satisfies the following conditions:
    \begin{enumerate}
        \item[$(1)$] $G_{x,y}(t):=G(x,y,t)$ is even, continuous, increasing and convex in $t$, and for each $t\in \mathbb{R}, \ \ G(x,y,t) $ is measurable in $(x,y)$;
        \item[$(2)$]$\displaystyle{\lim\limits_{t\rightarrow 0}\frac{G_{x,y}(t)}{t}=0}$, for a.a $(x,y)\in \Omega\times\Omega;$
        \item[$(3)$]$\displaystyle{\lim\limits_{t\rightarrow \infty}\frac{G_{x,y}(t)}{t}=\infty}$, for a.a $(x,y)\in \Omega\times\Omega;$
        \item[$(4)$] $\displaystyle{G_{x,y}(t)>0}$, for all $t>0$ and $(x,y)\in \Omega\times\Omega$.
    \end{enumerate}
\end{dfn}
\begin{dfn}
We say that a generalized N-function $G_{x,y}$ satisfies the $\Delta_2$-condition if there exists $K > 0$ such that
$$G_{x,y}(2t)\leq KG_{x,y}(t),\ \ \text{for all}\ (x,y)\in \Omega\times\Omega\ \text{and all}\ t>0.$$
\end{dfn}
\begin{dfn}
For any generalized N-function $G_{x,y}$, the function
$\widetilde{G}_{x,y}:\Omega\times\Omega\times\mathbb{R}\longrightarrow
\mathbb{R}_+$ defined by
\begin{equation}\label{2eq50}
\widetilde{G}_{x,y}(t)=\widetilde{G}(x,y,t):=\sup_{\tau\geq 0}\left( t\tau-G_{x,y}(\tau)\right),\ \ \text{for all}\ (x,y)\in \Omega\times\Omega\ \text{and all}\ t>0
\end{equation}
 is called the complementary function of $G_{x,y}$.
\end{dfn}

The assumptions $(g_1)-(g_4)$ ensure that $G_{x,y}$ and its
complementary function $\widetilde{G}_{x,y}$ are  generalized
N-functions (see \cite{M32}).
\begin{rem}[see \cite{M30}]\label{rem1}
Assumption $(g_4)$, gives that
 $$g^-\leq \frac{\widehat{g}(x,t)t}{\widehat{G}(x,t)}\leq g^+\ \text{and}\ \widetilde{g}^-\leq \frac{\widetilde{\widehat{g}}(x,t)t}{\widetilde{\widehat{G}}(x,t)}\leq \widetilde{g}^+,\ \text{for all}\ x\in \Omega\ \text{and all}\ t>0 $$
 where $\displaystyle{\widetilde{g}^-=\frac{g^-}{g^--1}}$ and $\displaystyle{\widetilde{g}^+=\frac{g^+}{g^+-1}}$.
 Moreover, $G_{x,y}$, $\widehat{G}_x$  and  $\widetilde{\widehat{G}}_x$ satisfy the $\Delta_2$-condition.
 \end{rem}
In view of definition of the complementary function
$\widetilde{G}_{x,y}$, we have the following Young's type
inequality:
\begin{equation}\label{2eq95}
    \tau\sigma\leq G_{x,y}(\tau)+\widetilde{G}_{x,y}(\sigma),\ \text{for all}\ (x,y)\in \Omega\times\Omega\ \text{and all}\ \tau,\sigma\geq0.
\end{equation}
\subsection{Musielak-Orlicz spaces}
Let $G_{x,y}$ be a generalized N-function. In correspondence to $\widehat{G}_x=G_{x,x}$ and an open subset $\Omega$ of $\mathbb{R}^d$, the Musielak-Orlicz space is defined as follows
$$L^{\widehat{G}_x}(\Omega):=\left\lbrace u:\Omega\longrightarrow \mathbb{R}\ \text{measurable :}\ J_{\widehat{G}_x}(\lambda u)<+\infty,\ \ \text{for some}\ \lambda>0\right\rbrace,$$
where
\begin{equation}\label{2eq3}
  J_{\widehat{G}_x}(u):= \int_{\Omega} \widehat{G}_x(\vert u\vert)dx.
\end{equation}
The space $L^{\widehat{G}_x}(\Omega)$ is endowed with the
Luxemburg norm
\begin{equation}\label{2eq4}
  \Vert u\Vert_{L^{\widehat{G}_x}(\Omega)}:=\inf\left\lbrace \lambda:\ J_{\widehat{G}_{x}}\left(\frac{u}{\lambda}\right)\leq 1\right\rbrace.
\end{equation}
We would like to mention that our assumptions $(g_1)-(g_4)$ ensure
that $\displaystyle{\left(L^{\widehat{G}_x}(\Omega),\Vert
\cdot\Vert_{L^{\widehat{G}_x}(\Omega)}\right)}$ is a separable and
reflexive Banach space.\\
Now, we recall the following technical and important lemmas.
\begin{lemma}[see \cite{M1}]\label{lem1}
  Assume that the assumptions $(g_1)-(g_4)$ hold. Then, the function $\widehat{G}_x$  and  $\Tilde{\widehat{G}}_x$ satisfy the following properties:
  \begin{enumerate}
   \item[$(1)$] $\displaystyle{\min\left\lbrace \tau^{g^-},\tau^{g^+}\right\rbrace G_{x,y}(t)\leq G_{x,y}(\tau t)\leq \max\left\lbrace \tau^{g^-},\tau^{g^+}\right\rbrace G_{x,y}(t)}$, for all $x \in \Omega$ and $\tau,t>0$;
      \item[$(2)$] $\displaystyle{\min\left\lbrace \tau^{g^-},\tau^{g^+}\right\rbrace\widehat{G}_x(t)\leq \widehat{G}_x(\tau t)\leq \max\left\lbrace \tau^{g^-},\tau^{g^+}\right\rbrace \widehat{G}_x(t)}$, for all $x \in \Omega$ and $\tau,t>0$;
      \item[$(3)$] $\displaystyle{\min\left\lbrace \tau^{\Tilde{g}^-},\tau^{\Tilde{g}^+}\right\rbrace\Tilde{\widehat{G}}_x(t)\leq \Tilde{\widehat{G}}_x(\tau t)\leq \max\left\lbrace \tau^{\Tilde{g}^-},\tau^{\Tilde{g}^+}\right\rbrace \Tilde{\widehat{G}}_x(t)}$, for all $x \in \Omega$ and $\tau,t>0$;
      \item[$(4)$] $\displaystyle{\min\left\lbrace \Vert u\Vert_{L^{\widehat{G}_x}(\Omega)}^{g^-},\Vert u\Vert_{L^{\widehat{G}_x}(\Omega)}^{g^+}\right\rbrace
      \leq J_{\widehat{G}_x}(u)\leq \max\left\lbrace \Vert u\Vert_{L^{\widehat{G}_x}(\Omega)}^{g^-},\Vert u\Vert_{L^{\widehat{G}_x}(\Omega)}^{g^+}\right\rbrace }$, for all $u\in L^{\widehat{G}_x}(\Omega)$;
      \item[$(5)$] $\displaystyle{\min\left\lbrace \Vert u\Vert_{L^{\Tilde{\widehat{G}}_x}(\Omega)}^{g^-},\Vert u\Vert_{L^{\Tilde{\widehat{G}}_x}(\Omega)}^{g^+}
      \right\rbrace\leq J_{\Tilde{\widehat{G}}_x}(u)\leq \max\left\lbrace \Vert u\Vert_{L^{\Tilde{\widehat{G}}_x}(\Omega)}^{g^-},
      \Vert u\Vert_{L^{\Tilde{\widehat{G}}_x}(\Omega)}^{g^+}\right\rbrace} $, for all $u\in L^{\Tilde{\widehat{G}}_x}(\Omega)$, where
       $\displaystyle{J_{\tilde{\widehat{G}}_x}(u):= \int_{\Omega} \tilde{\widehat{G}}_x(\vert u\vert)dx}$.
  \end{enumerate}
   \end{lemma}
  \begin{lemma}[see \cite{M32}]\label{lem3}
     Let $\widehat{G}_{x}$ be a generalized N-function. Then, we have
     $$ \Vert u\Vert_{L^{\widehat{G}_x}(\Omega)}\leq J_{\widehat{G}_x}(u)+1,\ \ \text{for all}\ u\in L^{\widehat{G}_x}(\Omega).$$
\end{lemma}
As a consequence of \eqref{2eq95}, we have the following lemma:
\begin{lemma}[H\"older's type
inequality]\label{lemhol}
  Let $\Omega$ be an open subset of $\mathbb{R}^d$. Let $\widehat{G}_x$ a generalized N-function and $\widetilde{\widehat{G}}_x$ its complementary function, then
  \begin{equation}\label{2eq96}
     \left\vert \int_{\Omega} uv dx \right\vert \leq 2 \Vert u\Vert_{L^{\widehat{G}_x}(\Omega)} \Vert v\Vert_{L^{\widetilde{\widehat{G}}_x}(\Omega)},\ \text{for all}\ u\in L^{\widehat{G}_x}(\Omega)\ \text{and}\ v\in L^{\widetilde{\widehat{G}}_x}(\Omega).
  \end{equation}
\end{lemma}
Proceeding as in \cite[Lemma 3.4]{M50}, we get the following result:
\begin{lemma}\label{lem55} Let $G_{x,y}$ be a generalized N-function satisfying the assumptions $(g_1)-(g_4)$. Then, we have
$$\left(g_{x,y}(\tau)-g_{x,y}(\sigma)\right)(\tau -\sigma)\geq 4 G_{x,y}\left(\frac{\tau-\sigma}{2}\right),\ \text{for all}\ \tau,\sigma \in \mathbb{R}\setminus\lbrace 0\rbrace \ \text{and}\ x,y\in \mathbb{R}^d.$$
\end{lemma}
\subsection{Fractional Musielak-Sobolev spaces}
Let $G_{x,y}$ be a generalizd N-function, $s\in(0,1)$ and $\Omega$ an open subset of $\mathbb{R}^d$. The fractional Musielak-Sobolev space is defined as follows

$$W^{s,G_{x,y}}(\Omega):=\left\lbrace u\in L^{\widehat{G}_x}(\Omega):\ J_{s,G_{x,y}}(\lambda u)\leq +\infty,\ \text{for some}\ \lambda>0\right\rbrace,$$
where
\begin{equation}\label{2eq1}
 J_{s,G_{x,y}}(u):=\int_{\Omega}\int_{\Omega}G_{x,y}\left(\frac{u(x)-u(y)}{\vert x-y\vert^s}\right)\frac{dxdy}{\vert x-y\vert^d}.
\end{equation}
The space $W^{s,G_{x,y}}(\Omega)$ is endowed with the norm
\begin{equation}\label{2eq2}
  \Vert u\Vert_{W^{s,G_{x,y}}(\Omega)}:=\Vert u\Vert_{L^{\widehat{G}_x}(\Omega)}+\left[u\right]_{s,G_{x,y}},\ \ \text{for all}\ u\in W^{s,G_{x,y}}(\Omega),
\end{equation}
with $\left[u\right]_{s,G_{x,y}}$is the so called $(s,G_{x,y})$-Gagliardo seminorm defined by
\begin{equation}\label{2eq5}
\left[u\right]_{s,G_{x,y}}:=\inf\left\lbrace \lambda:\ J_{s,G_{x,y}}\left(\frac{u}{\lambda}\right)\leq 1\right\rbrace.
\end{equation}
\begin{rem}\label{rem2}
    Since assumption $(g_4)$ implies that the functions $\widehat{G}_x$  and  $\widetilde{\widehat{G}}_x$ satisfy the $\Delta_2$-condition, the space
$W^{s,G_{x,y}}(\Omega)$ is a  reflexive and separable Banach space.
\end{rem}
Let $\widehat{G}_x$ be defined as in \eqref{2eq9}, the assumptions $(g_1)-(g_3)$, confirm that, for each $x\in \Omega$, $\widehat{G}_x:\mathbb{R}_+\longrightarrow \mathbb{R}_+$  is an increasing homeomorphism. Hence, the inverse function $\widehat{G}_x^{-1}$ of $\widehat{G}_x$  exists.\\
Throughout this paper, we  assume the following condition:
\begin{equation}\label{g5}
  \int_{0}^{1} \frac{\widehat{G}^{-1}_x(\tau)}{\tau^{\frac{d+s}{d}}} d\tau<+\infty\ \ \text{and}\ \ \int_{1}^{+\infty} \frac{\widehat{G}^{-1}_x(\tau)}{\tau^{\frac{d+s}{d}}} d\tau=+\infty,\ \ \text{for all}\ x\in \Omega.\tag{$g_5$}
\end{equation}
Now, we define the inverse of an important function which is the Musielak-Sobolev conjugate function of $\widehat{G}_x$, denoted by $\widehat{G}_x^*$, as follows:
\begin{equation}\label{2eq10}
\left(\widehat{G}_x^*\right)^{-1}(t):= \int_{0}^{t} \frac{\widehat{G}^{-1}_x(\tau)}{\tau^{\frac{d+s}{d}}} d\tau,\ \ \text{for all}\ x\in \Omega\ \text{and}\ t\geq 0.
\end{equation}
\begin{lemma}[see \cite{M1}]\label{lem2}
  Assume that  assumptions $(g_1)-(g_5)$ hold with $g^-,g^+\in (1,\frac{d}{s})$ and $s\in (0,1)$. Then, we have the following properties:
  \begin{enumerate}
   \item[$(1)$] $\displaystyle{\min\left\lbrace \left[u\right]_{s,G_{x,y}}^{g^-},\left[u\right]_{s,G_{x,y}}^{g^+}\right\rbrace \leq J_{s,G_{x,y}}(u)\leq \max\left\lbrace \left[u\right]_{s,G_{x,y}}^{g^-},\left[u\right]_{s,G_{x,y}}^{g^+}\right\rbrace}$, for all $u \in W^{s,G_{x,y}}(\Omega)$;
      \item[$(2)$] $\displaystyle{\min\left\lbrace \tau^{g^-_*},\tau^{g^+_*}\right\rbrace\widehat{G}_x^*(t)\leq \widehat{G}_x^*(\tau t)\leq \max\left\lbrace \tau^{g^-_*},\tau^{g^+_*}\right\rbrace \widehat{G}_x^*(t)}$, for all $x \in \Omega$ and $\tau,t>0$;
      \item[$(3)$] $\displaystyle{\min\left\lbrace \Vert u\Vert_{L^{\widehat{G}_x^*}(\Omega)}^{g^-_*},\Vert u\Vert_{L^{\widehat{G}^*_x}(\Omega)}^{g^+_*}
      \right\rbrace\leq J_{\widehat{G}^*_x}(u)\leq \max\left\lbrace \Vert u\Vert_{L^{\widehat{G}^*_x}(\Omega)}^{g^-_*},
      \Vert u\Vert_{L^{\widehat{G}^*_x}(\Omega)}^{g^+_*}\right\rbrace }$, for all $u\in L^{\widehat{G}^*_x}(\Omega)$.
  \end{enumerate}
  where $\displaystyle{g^-_*=\frac{dg^-}{d-sg^-}}$, $\displaystyle{g^+_*:=\frac{dg^+}{d-sg^+}}$, and $\displaystyle{J_{\widehat{G}^*_x}(u):= \int_{\Omega} \widehat{G}^*_x(\vert u\vert)dx}$.
\end{lemma}
\begin{dfn}
  We say that a generalized N-function $G_{x,y}$ satisfies the fractional boundedness condition if  there exist $C_1,C_2>0$ such that
  \begin{equation}\label{bf}
      C_1\leq G_{x,y}(1)\leq C_2,\ \ \text{for all}\ (x,y)\in \Omega\times\Omega.\tag{$\mathcal{B}_f$}
  \end{equation}
\end{dfn}
\begin{dfn}
  Let $\widehat{A}_x$ and $\widehat{B}_x$ be two generalized  N-functions. We say that $\widehat{A}_x$ essentially grows more slowly than $\widehat{B}_x$ near
infinity, and we write  $\widehat{A}_x\prec\prec \widehat{B}_x$, if for all $k > 0$, we have
$$\lim\limits_{t\rightarrow +\infty}\frac{\widehat{A}_x(kt)}{\widehat{B}_x(t)}=0,\ \ \text{uniformly in }\ x\in \Omega.$$
\end{dfn}
\begin{thm}[see \cite{M8,M9}]\label{thm2}
  Let $s\in (0,1)$, $G_{x,y}$ a generalized N-function satisfying $(g_1)-(g_4)$, and $\Omega$
 a bounded domain in $\mathbb{R}^d$ with $C^{0;1}$-regularity and bounded boundary.
 \begin{enumerate}
     \item[$(1)$] If \eqref{bf} and \eqref{g5} hold, the embedding $W^{s,G_{x,y}}(\Omega)\hookrightarrow L^{\widehat{G}_x^*}(\Omega)$ is continuous.
     \item[$(2)$]Moreover, for any generalized N-function $\widehat{A}_x$ such that $\widehat{A}_x\prec\prec \widehat{G}_x^*$, the embedding $W^{s,G_{x,y}}(\Omega)\hookrightarrow L^{\widehat{A}_x}(\Omega)$ is compact.
 \end{enumerate}
\end{thm}
\begin{lemma}\label{Aux}
    Let $\Omega$ be an open subset of $\mathbb{R}^d$. Let $\widehat{G}_x$ be a generalized N-function  satisfying the assumptions $(g_1)-(g_4)$ and $\widehat{G}_x^*$ its Musielak-Sobolev conjugate
function. Then, there exists a generalized N-function $\widehat{R}_x$ satisfies the following assertions:
\begin{enumerate}
    \item[$(1)$] $\displaystyle{1<r^-\leq\frac{\widehat{r}_x(t)t}{\widehat{R}_x(t)}\leq r^+< \frac{g^-_*}{g^+}},$ for all $x\in \Omega$ and $t>0$, where $\displaystyle{\widehat{R}_x(t):=\int_{0}^t\widehat{r}_x(s)ds}$;
    \item[$(2)$]the condition \eqref{bf};
    \item[$(3)$]$\displaystyle{\widehat{R}_x\circ \widehat{G}_x\prec
 \widehat{G}_x^*,\ \ \text{for all}\ x\in\Omega.}$
\end{enumerate}
\end{lemma}
\begin{proof}
Let us define the generalized N-function $\widehat{R}_x:\Omega\times \mathbb{R}\longrightarrow \mathbb{R}$ by$$\widehat{R}_x(t):=\frac{1}{p(x)}\vert t\vert^{p(x)},\ \ \text{for all}\ (x,t)\in \Omega\times \mathbb{R},$$ where $p$ is a real-valued function satisfying
$$1<r^- \leq p(x)\leq r^+<\frac{g^-_*}{g^+},\ \ \text{for all}\ x\in \Omega.$$
It's clear that $\widehat{R}_x$ is a generalized N-function and verifies the assertions $(1)$ and $(2)$. Assertion $(3)$ is a consequence from assertion $(1)$. Indeed: Using Lemmas \ref{lem1} and \ref{lem2}, for all $k>0$ and $t>1$, we get
\begin{equation}\label{2eq40}
    \widehat{G}_x(kt)\leq \widehat{G}_x(k)t^{g^+}\ \ \text{and}\ \ \widehat{G}_x^*(1)t^{g^-_*}\leq \widehat{G}_x^*(t),\ \ \text{for all}\ x\in \Omega.
\end{equation}
From assertion $(1)$, we can see that
$$\min\left\lbrace \tau^{r^-},\tau^{r^+}\right\rbrace\widehat{R}_x(t)\leq \widehat{R}_x(\tau t)\leq \max\left\lbrace \tau^{r^-},\tau^{r^+}\right\rbrace \widehat{R}_x(t),\ \text{for all}\ x \in \Omega\ \text{and}\  \tau,t>0.$$
It follows, by \eqref{2eq40}, that
\begin{equation}\label{2eq41}
   \widehat{R}_x\left(\widehat{G}_x(kt)\right) \leq \widehat{R}_x\left(\widehat{G}_x(k)t^{g^+}\right) \leq \widehat{R}_x\left(\widehat{G}_x(k)\right)t^{g^+r^+},\ \ \text{for all}\ k>0,\ t>1\ \text{and}\ x\in \Omega.
\end{equation}
Putting together \eqref{2eq40} and \eqref{2eq41}, we obtain
$$\frac{\widehat{R}_x\left(\widehat{G}_x(kt)\right) }{\widehat{G}_x^*(t)}\leq\frac{\widehat{R}_x\left(\widehat{G}_x(k)\right)t^{g^+r^+}}{\widehat{G}_x^*(1)t^{g^-_*}},\ \ \text{for all}\ k>0,\ t>1\ \text{and}\ x\in \Omega.$$
Since $g^+r^+<g^-_*$, we conclude that
$$\lim\limits_{t\rightarrow +\infty}\frac{\widehat{R}_x\left(\widehat{G}_x(kt)\right) }{\widehat{G}_x^*(t)}=0,\ \ \text{for all}\ k>0,\ \text{and}\ x\in \Omega.$$
Thus, the proof is complete.
\end{proof}
\subsection{Fractional Musielak $g_{x,y}$-Laplacian}
\begin{dfn}
 Let $G_{x,y}$ be a generalized N-function and $s\in(0,1)$. The fractional Musielak $g_{x,y}$-Laplacian is defined by
 \begin{align*}
(-\Delta)^s_{g_{x,y}}u(x):&=\text{p.v.}\int_{\mathbb{R}^d}a_{x,y}\left(\frac{\vert u(x)-u(u)\vert}{\vert x-y\vert^{s}}\right)\frac{\vert u(x)-u(y)\vert}{\vert x-y\vert^{s}}\frac{dy}{\vert x-y\vert^{d+s}}\\
& =\text{p.v.}\int_{\mathbb{R}^d}g_{x,y}\left(\frac{\vert u(x)-u(u)\vert}{\vert x-y\vert^{s}}\right)\frac{dy}{\vert x-y\vert^{d+s}},\ \ \text{for all}\ x\in \mathbb{R}^d
 \end{align*}
 where p.v. is a commonly used abbreviation for "in the principle value sense" and $\displaystyle{G_{x,y}(t)=\int_0^{\vert t\vert}g_{x,y}(\tau)d\tau}$.
\end{dfn}
Under the assumptions $(g_1)-(g_3)$, the operator $(-\Delta)^s_{g_{x,y}}$ is well defined between $W^{s,G_{x,y}}(\mathbb{R}^d)$ and its topological dual space $\left(W^{s,G_{x,y}}(\mathbb{R}^d)\right)^*$. According to \cite{M7}, we have that
\begin{align}\label{2eq85}
  \langle (- \Delta)^s_{g_{x,y}}u, v\rangle&=\int_{\mathbb{R}^d}\int_{\mathbb{R}^d}g_{x,y}\left( \frac{u(x)-u(y)}{\vert x-y\vert^{s}}\right) \frac{v(x)-v(y)}{\vert x-y\vert^{d+s}}dxdy\nonumber\\ &=\int_{\mathbb{R}^d}\int_{\mathbb{R}^d}G^{'}_{x,y}\left( \frac{u(x)-u(y)}{\vert x-y\vert^{s}}\right) \frac{v(x)-v(y)}{\vert x-y\vert^{d+s}}dxdy\nonumber\\
  & =\langle J^{'}_{s,G_{x,y}}(u),v\rangle,\ \ \text{for all}\ u,v\in W^{s,G_{x,y}}(\mathbb{R}^d)
\end{align}
where $\langle \cdot,\cdot\rangle$ is the duality brackets for the pair $\left(\left(W^{s,G_{x,y}}(\mathbb{R}^d)\right)^*,W^{s,G_{x,y}}(\mathbb{R}^d)\right)$.\\

Next, proceeding as in \cite[Theorem 3.14]{M1}, we obtain the following result:
\begin{prop}\label{s+}
Assume that $(g_1)-(g_4)$ holds. Then $J^{'}_{s,G_{x,y}}$ satisfies the $(S_+)$, that is, for every sequence $\lbrace u_n\rbrace_{n\in \mathbb{N}}\subset W^{s,G_{x,y}}(\mathbb{R}^d)$ such that $u_n\rightharpoonup u$ in  $W^{s,G_{x,y}}(\mathbb{R}^d)$
and
$$\limsup_{n\rightarrow +\infty}\langle J^{'}_{s,G_{x,y}}(u_n),u_n-u\rangle\leq 0,$$
we have that
$$u_n\longrightarrow u\ \ \text{in}\ W^{s,G_{x,y}}(\mathbb{R}^d).$$
\end{prop}
\section{Proof of Theorem \ref{thm1}}
In this section, we prove a continuous embedding result for the
fractional Musielak-Sobolev space in $\mathbb{R}^d$.\\
First, we establish some notions and technical lemmas which are useful in the proof of Theorem \ref{thm1}.\\

Let $\Omega$ be an open subset of $\mathbb{R}^{d}$. For each $u\in W^{s,G_{x,y}}(\Omega)$, we set
\begin{equation}\label{2eq7}
    \rho(u):=J_{\widehat{G}_x}(u)+J_{s,G_{x,y}}(u)
\end{equation}
and
\begin{equation}\label{2eq6}
\Vert u\Vert_{(\Omega)}:=\inf\bigg{\{}\lambda>0\ :\ \rho\bigg{(}\frac{u}{\lambda}\bigg{)}\leq1\bigg{\}}.
\end{equation}

\begin{rem}\label{norm}
For all $u\in W^{s,G_{x,y}}(\Omega)$,   Fatou's lemma gives that $$\rho\bigg{(}\frac{u}{\Vert u\Vert_{(\Omega)}}\bigg{)}\leq1.$$
\end{rem}

\begin{lemma}\label{3.5}
 $\Vert \cdot\Vert_{(\Omega)}$ is a norm in $W^{s,G_{x,y}}(\Omega)$. Moreover, $\Vert \cdot\Vert_{(\Omega)}$ and $\Vert\cdot\Vert_{W^{s,G_{x,y}}(\Omega)}$ are equivalents, with the relation
 \begin{equation}\label{equivalence}
   \frac{1}{2}\Vert u\Vert_{W^{s,G_{x,y}}(\Omega)}\leq \Vert u\Vert_{(\Omega)}\leq 2\Vert u\Vert_{W^{s,G_{x,y}}(\Omega)},\ \text{for all}\ u\in W^{s,G_{x,y}}(\Omega).
 \end{equation}
\end{lemma}

\begin{proof}
First, Let's  prove that $\Vert \cdot\Vert_{(\Omega)}$ is a norm in $W^{s,G_{x,y}}(\Omega)$. To this end, we show that  $\Vert \cdot\Vert$  verifies the well-known three axioms of the norm: \\
 \noindent $(\mathbf{i})$ It is clear that, if $\Vert u\Vert_{(\Omega)}=0$, then $u=0,\ a.a$.\\
  $(\mathbf{ii})$ For  each $\alpha\in \mathbb{R}$, we have
  \begin{align*}
    \Vert \alpha u\Vert_{(\Omega)} & =\inf\bigg{\{}\lambda>0:\ \rho\bigg{(}\frac{\alpha u}{\lambda}\bigg{)}\leq1\bigg{\}}=
     \inf\bigg{\{}|\alpha|\lambda>0:\ \rho\bigg{(}\frac{u}{\lambda}\bigg{)}\leq1\bigg{\}}\\
     &=|\alpha|\inf\bigg{\{}\lambda>0:\ \rho\bigg{(}\frac{u}{\lambda}\bigg{)}\leq1\bigg{\}}
                               =|\alpha| \Vert u\Vert_{(\Omega)}.
  \end{align*}

\noindent$(\mathbf{iii})$ Finally for the triangle inequality, let $u,v\in W^{s,G_{x,y}}(\Omega)$, we compute
  \begin{align*}
    \rho\bigg{(}\frac{u+v}{\Vert u\Vert_{(\Omega)}+\Vert v\Vert_{(\Omega)}}\bigg{)} & =\rho\bigg{(}\frac{\Vert u\Vert_{(\Omega)}}{\Vert u\Vert_{(\Omega)}+\Vert v\Vert_{(\Omega)}}\frac{u}{\Vert u\Vert_{(\Omega)}}
    +\frac{\Vert v\Vert_{(\Omega)}}{\Vert u\Vert_{(\Omega)}+\Vert v\Vert_{(\Omega)}}\frac{v}{\Vert v\Vert_{(\Omega)}}\bigg{)} \\ &\leq\rho\bigg{(}\frac{\Vert u\Vert_{(\Omega)}}{\Vert u\Vert_{(\Omega)}+\Vert v\Vert_{(\Omega)}}\frac{u}{\Vert u\Vert_{(\Omega)}}\bigg{)}
    +\rho\bigg{(}\frac{\Vert v\Vert_{(\Omega)}}{\Vert u\Vert_{(\Omega)}+\Vert v\Vert_{(\Omega)}}\frac{v}{\Vert v\Vert_{(\Omega)}}\bigg{)}\\
    &\leq\frac{\Vert u\Vert_{(\Omega)}}{\Vert u\Vert_{(\Omega)}+\Vert v\Vert_{(\Omega)}}\rho\bigg{(}\frac{u}{\Vert u\Vert_{(\Omega)}}\bigg{)}
    +\frac{\Vert v\Vert_{(\Omega)}}{\Vert u\Vert_{(\Omega)}+\Vert v\Vert_{(\Omega)}}\rho\bigg{(}\frac{v}{\Vert v\Vert_{(\Omega)}}\bigg{)}\\
    &\leq1.
  \end{align*}
  Thus, $$\Vert u+v\Vert_{(\Omega)}\leq \Vert u\Vert_{(\Omega)}+\Vert v\Vert_{(\Omega)}\ \ ,\ \text{for all}\ u,v\in W^{s,G_{x,y}}(\Omega).$$
Second, we show the inequality \eqref{equivalence}. On account of this, we prove the left side of the inequality \eqref{equivalence}. In this way, by using Remark \ref{norm} anyone can check that for each $u\in W^{s,G_{x,y}}(\Omega)$, we have

  $$J_{\widehat{G}_x}\bigg{(}\frac{u}{\Vert u\Vert_{(\Omega)}}\bigg{)}\leq\rho\bigg{(}\frac{u}{\Vert u\Vert_{(\Omega)}}\bigg{)}\leq1\ \ \ \text{and}\ \ \
  J_{s,G_{x,y}}\bigg{(}\frac{u}{\Vert u\Vert_{(\Omega)}}\bigg{)}\leq\rho\bigg{(}\frac{u}{\Vert u\Vert_{(\Omega)}}\bigg{)} \leq1.$$
It follows, by \eqref{2eq4} and \eqref{2eq5}, that $$\|u\|_{L^{\widehat{G}_x}(\Omega)}\leq \Vert u\Vert_{(\Omega)}\ \ \text{and}\ \ [u]_{s,G_{x,y}}\leq \Vert u\Vert_{(\Omega)}.$$
  Therefore,
  $$\frac{1}{2}\Vert u\Vert_{W^{s,G_{x,y}}(\Omega)}\leq \Vert u\Vert_{(\Omega)},\ \text{for all}\ u\in W^{s,G_{x,y}}(\Omega).$$
  For the right side of the inequality  \eqref{equivalence}, we have
  \begin{align*}
    \rho\bigg{(}\frac{u}{2\|u\|_{W^{s,G_{x,y}}(\Omega)}}\bigg{)} & \leq \frac{1}{2}J_{\widehat{G}_x}\bigg{(}\frac{u}{\|u\|_{W^{s,G_{x,y}}(\Omega)}}\bigg{)}+\frac{1}{2}J_{s,G_{x,y}}\bigg{(}\frac{u}{\|u\|_{W^{s,G_{x,y}}(\Omega)}}\bigg{)}\\
    &\leq \frac{1}{2}J_{\widehat{G}_x}\bigg{(}\frac{u}{\|u\|_{L^{G_{x}}(\Omega)}}\bigg{)}+\frac{1}{2}J_{s,G_{x,y}}\bigg{(}\frac{u}{[u]_{s,G_{x,y}}(\Omega)}\bigg{)}\\
&\leq \frac{1}{2}+\frac{1}{2}=1.
  \end{align*}
  Thus,  $$\Vert u\Vert_{(\Omega)}\leq 2\Vert u\Vert_{W^{s,G_{x,y}}(\Omega)},\ \text{for all}\ u\in W^{s,G_{x,y}}(\Omega).$$
  This ends the proof Lemma \ref{3.5}.
\end{proof}

\begin{lemma}\label{rho}
Assume that the assumptions $(g_1)-(g_4)$ hold. For each $u\in W^{s,G_{x,y}}(\Omega)$, we have
$$\min\left\lbrace \Vert u\Vert_{(\Omega)}^{g^-},\Vert u\Vert_{(\Omega)}^{g^+}\right\rbrace\leq \rho(u)\leq \max\left\lbrace \Vert u\Vert_{(\Omega)}^{g^-},\Vert u\Vert_{(\Omega)}^{g^+}\right\rbrace.$$
\end{lemma}

\begin{proof}
 Let $u\in W^{s,G_{x,y}}(\Omega)$. If $\Vert u\Vert_{(\Omega)}=0$, then $u=0$ a.a. and $\rho(u)=0$. Thus,
 $$\min\left\lbrace \Vert u\Vert_{(\Omega)}^{g^-},\Vert u\Vert_{(\Omega)}^{g^+}\right\rbrace\leq \rho(u)\leq \max\left\lbrace \Vert u\Vert_{(\Omega)}^{g^-},\Vert u\Vert_{(\Omega)}^{g^+}\right\rbrace.$$
If $\Vert u\Vert_{(\Omega)}\neq 0$, then,  Lemmas \ref{lem1}, \ref{lem2} and Remark \ref{norm}, give that
  \begin{align*}
    \rho\bigg{(}u\bigg{)} & =
    \int_{\Omega}\widehat{G}_x\bigg{(}\frac{\Vert u\Vert_{(\Omega)}}{\Vert u\Vert_{(\Omega)}}u\bigg{)}dx+
\int_{\Omega}\int_{\Omega}G_{x,y}\bigg{(}\frac{\Vert u\Vert_{(\Omega)}}{\Vert u\Vert_{(\Omega)}}\frac{u(x)-u(y)}{\vert x-y\vert^s}\bigg{)}\frac{dxdy}{|x-y|^{d}}\\
    &\leq \max\left\lbrace \Vert u\Vert_{(\Omega)}^{g^-},\Vert u\Vert_{(\Omega)}^{g^+}\right\rbrace\bigg{[}\int_{\Omega}\widehat{G}_x\bigg{(}\frac{1}{\Vert u\Vert_{(\Omega)}}u\bigg{)}dx+
\int_{\Omega}\int_{\Omega}G_{x,y}\bigg{(}\frac{1}{\Vert u\Vert_{(\Omega)}}\frac{u(x)-u(y)}{\vert x-y\vert^s}\bigg{)}\frac{dxdy}{|x-y|^{d}}\bigg{]}\\
    &= \max\left\lbrace \Vert u\Vert_{(\Omega)}^{g^-},\Vert u\Vert_{(\Omega)}^{g^+}\right\rbrace\rho\bigg{(}\frac{u}{\Vert u\Vert_{(\Omega)}}\bigg{)}\\
    &\leq\max\left\lbrace \Vert u\Vert_{(\Omega)}^{g^-},\Vert u\Vert_{(\Omega)}^{g^+}\right\rbrace.
  \end{align*}
  Let $0<\sigma<\Vert u\Vert_{(\Omega)}$, by \eqref{2eq6},
  $$\rho\bigg{(}\frac{u}{\sigma}\bigg{)}>1.$$
  It follows, by
  Lemma \ref{lem1}, that
    \begin{align*}
    \rho\bigg{(}u\bigg{)} & =
    \int_{\Omega}\widehat{G}_x\bigg{(}\frac{\sigma}{\sigma}u\bigg{)}dx+
\int_{\Omega}\int_{\Omega}G_{x,y}\bigg{(}\frac{\sigma}{\sigma}\frac{u(x)-u(y)}{\vert x-y\vert^s}\bigg{)}\frac{dxdy}{|x-y|^{d}}\\
    &\geq \min\left\lbrace \sigma^{g^-},\sigma^{g^+}\right\rbrace\bigg{[}\int_{\Omega}\widehat{G}_x\bigg{(}\frac{1}{\sigma}u\bigg{)}dx+
\int_{\Omega}\int_{\Omega}G_{x,y}\bigg{(}\frac{1}{\sigma}\frac{u(x)-u(y)}{\vert x-y\vert^s}\bigg{)}\frac{dxdy}{|x-y|^{d}}\bigg{]}\\
    &= \min\left\lbrace \sigma^{g^-},\sigma^{g^+}\right\rbrace\rho\bigg{(}\frac{u}{\sigma}\bigg{)}\\
    &\geq\min\left\lbrace \sigma^{g^-},\sigma^{g^+}\right\rbrace.
  \end{align*}
  Letting $\sigma\longrightarrow \Vert u\Vert_{(\Omega)}$ in the above inequality, we obtain
  $$\min\left\lbrace \Vert u\Vert_{(\Omega)}^{g^-},\Vert u\Vert_{(\Omega)}^{g^+}\right\rbrace\leq \rho(u).$$
 Thus, the proof is complete.
\end{proof}
\begin{proof}[{\bf Proof of Theorem \ref{thm1}}] $(1)$
  Let $u\in W^{s,G_{x,y}}(\mathbb{R}^d)$ such that $\Vert u\Vert_{(\mathbb{R}^d)}=1$. By Lemma
  \ref{rho},
  \begin{equation}\label{2eq8}
  \rho(u)=1.
   \end{equation}
  We set $B_{i}:=\{x\in\mathbb{R}^{d}:\  i\leq|x|<i+1\}$, $i\in\mathbb{N}$, such that $\displaystyle{\mathbb{R}^{d}= \bigcup_{i\in\mathbb{N}}B_{i}}$ and $B_{i}\cap B_{j}\neq\emptyset$, for any $i\neq j$, and $\displaystyle{C:=\sup\bigg{\{}\int_{\mathbb{R}^{d}}\widehat{G}_{x}^*(u(x))dx:\ u\in W^{s,G_{x,y}}(\mathbb{R}^{d}),\Vert u\Vert_{(\mathbb{R}^d)}=1\bigg{\}}.}$\\
In light of  \eqref{2eq7} and \eqref{2eq8}, we see that
  \begin{align}\label{tilder}
 \int_{B_{i}}\widehat{G}_x(u(x))dx+ \int_{B_{i}} \int_{B_{i}}G_{x,y}\bigg{(}\frac{u(x)-u(y)}{|x-y|^{s}}\bigg{)}\frac{dxdy}{\vert x-y\vert^{d}} & \leq \rho(u),\ \text{for all}\ i\in \mathbb{N}.
   \end{align}
  \textbf{Claim:} $\displaystyle{C<\infty.}$\\
Indeed, in view of  \eqref{equivalence} and Theorem \ref{thm2}, there is a constant $C_{0}>0$ such that  $$\|u\|_{L^{\widehat{G}_{x}^*}(B_i)}\leq C_{0}\|u\|_{W^{s,G_{x,y}}(B_i)}\leq2C_{0}\|u\|_{(B_i)}\leq 2C_{0}\|u\|_{(\mathbb{R}^d)}=2C_{0},\ \text{for all}\ i\in\mathbb{N}.$$

Let $i\in\mathbb{N}$, we distinguish two cases:\\
  {\bf Cas 1:} If $1\leq\|u\|_{L^{\widehat{G}_{x}^*}(B_i)}\leq2C_{0}$, then, by \eqref{equivalence}, \eqref{tilder} and  Lemmas \ref{lem1}, \ref{lem2}, and \ref{rho}, we have
  \begin{align}\label{en1}
 \int_{B_{i}}\widehat{G}_{x}^*(u(x))dx& \leq \|u\|_{L^{\widehat{G}_{x}^*}(B_i)}^{g_*^+}\leq (2C_{0})^{g^+_*}\Vert u\Vert_{(B_i)}^{g^+_*}\nonumber\\
       &\leq (2C_{0})^{g_*^+}\bigg{(} \int_{B_{i}}\widehat{G}_x(u(x))dx+ \int_{B_{i}} \int_{B_{i}}G_{x,y}\bigg{(}\frac{u(x)-u(y)}{|x-y|^{s}}\bigg{)}\frac{dxdy}{\vert x-y\vert^{d}}  \bigg{)}^{\frac{g_*^+}{g^+}}\nonumber\\
       &\leq (2C_{0})^{g_*^+}\bigg{(} \int_{B_{i}}\widehat{G}_x(u(x))dx+ \int_{B_{i}} \int_{B_{i}}G_{x,y}\bigg{(}\frac{u(x)-u(y)}{|x-y|^{s}}\bigg{)}\frac{dxdy}{\vert x-y\vert^{d}} \bigg{)}.
  \end{align}
 {\bf Cas 2:} If $\|u\|_{L^{\widehat{G}_{x}^*}(B_i)}<1$, then, also by \eqref{equivalence}, \eqref{tilder} and  Lemmas \ref{lem1}, \ref{lem2}, and \ref{rho}, we have

 \begin{align}\label{en2}
\int_{B_{i}}\widehat{G}_{x}^*(u(x))dx& \leq \|u\|_{L^{\widehat{G}_{x}^*}(B_i)}^{g_*^-}\leq (2C_{0})^{g^-_*}\Vert u\Vert_{(B_i)}^{g^-_*}\nonumber\\
       &\leq (2C_{0})^{g_*^-}\bigg{(} \int_{B_{i}}\widehat{G}_x(u(x))dx+ \int_{B_{i}} \int_{B_{i}}G_{x,y}\bigg{(}\frac{u(x)-u(y)}{|x-y|^{s}}\bigg{)}\frac{dxdy}{\vert x-y\vert^{d}}  \bigg{)}^{\frac{g_*^-}{g^-}}\nonumber\\
       &\leq (2C_{0})^{g_*^-}\bigg{(} \int_{B_{i}}\widehat{G}_x(u(x))dx+ \int_{B_{i}} \int_{B_{i}}G_{x,y}\bigg{(}\frac{u(x)-u(y)}{|x-y|^{s}}\bigg{)}\frac{dxdy}{\vert x-y\vert^{d}} \bigg{)}.
  \end{align}
 Putting together \eqref{en1} and \eqref{en2}, we obtain \begin{align*}
         \int_{\mathbb{R}^{d}}\widehat{G}_{x}^*(u(x))dx & =\sum_{i\in\mathbb{N}}\int_{B_{i}}\widehat{G}_{x}^*(u(x))dx\\
        &\leq\left[(2C_{0})^{g_*^+}+(2C_{0})^{g^-_*}\right]\sum_{i\in\mathbb{N}} \bigg{(} \int_{B_{i}}\widehat{G}_x(u(x))dx+ \int_{B_{i}} \int_{B_{i}}G_{x,y}\bigg{(}\frac{u(x)-u(y)}{|x-y|^{s}}\bigg{)}\frac{dxdy}{\vert x-y\vert^{d}} \bigg{)}\\
        &= \left[(2C_{0})^{g_*^+}+(2C_{0})^{g^-_*}\right]\rho(u)=(2C_{0})^{g_*^+}+(2C_{0})^{g^-_*}.
       \end{align*}
      Therefore, $$C\leq(2C_{0})^{g_*^+}+(2C_{0})^{g^-_*}.$$
Thus, the proof of the claim.\\

Now, let $u\in W^{s,G_{x,y}}(\mathbb{R}^{d})\setminus\{0\}$ and $v:=\displaystyle\frac{u}{\|u\|_{(\mathbb{R}^d)}}$. Using Lemma \ref{lem3}, we infer that $$\|v\|_{L^{\widehat{G}_{x}^*}(\mathbb{R}^d)}\leq \int_{\mathbb{R}^{d}}\widehat{G}_{x}^*(v(x))dx+1\leq C+1.$$
 It follows that $$\|u\|_{L^{\widehat{G}_{x}^*}(\mathbb{R}^d)}\leq (C+1)\|u\|_{(\mathbb{R}^d)}.$$
 Hence, the embedding $W^{s,G_{x,y}}(\mathbb{R}^d)\hookrightarrow L^{\widehat{G}_x^*}(\mathbb{R}^d)$ is continuous.\\
 $(2)$ Let $\widehat{A}_x$ be a generalized N-function satisfying \eqref{2eq60} and \eqref{mla1b}. Then, we have  the following continuous embedding $$W^{s,G_{x,y}}(\mathbb{R}^{d})\hookrightarrow L^{\widehat{A}_x}(\mathbb{R}^{d}).$$
  Indeed: Using \eqref{2eq60} and \eqref{mla1b}, we can find $\delta , T>0$ such that
  $$\widehat{A}_x(t)\leq \widehat{G}_x(t),\ \text{for all}\ \vert t\vert \leq \delta\ \text{and all}\ x\in \mathbb{R}^d$$
  and
  $$\widehat{A}_x(t)\leq \widehat{G}^*_x(t),\ \text{for all}\ \vert t\vert \geq T\ \text{and all}\ x\in \mathbb{R}^d.$$
  It follows, from \eqref{cA}, \eqref{bf}, and Lemma \ref{lem1}, that for all $u\in W^{s,G_{x,y}}(\mathbb{R}^d)$, we have
  \begin{align}
      \int_{\mathbb{R}^d}\widehat{A}_{x}(u(x))dx&\leq \int_{\lbrace \vert u\vert\geq T\rbrace}\widehat{G}_{x}^*(u(x))dx+\int_{\lbrace \vert u\vert\leq \delta\rbrace}\widehat{G}_{x}(u(x))dx+\int_{\lbrace \delta<\vert u\vert<T\rbrace}\widehat{A}_{x}(u(x))dx\nonumber\\
      & \leq \int_{\lbrace \vert u\vert\geq T\rbrace}\widehat{G}_{x}^*(u(x))dx+\int_{\lbrace \vert u\vert\leq \delta\rbrace}\widehat{G}_{x}(u(x))dx+\int_{\lbrace \delta<\vert u\vert<T\rbrace}\widehat{A}_{x}(T)dx\nonumber\\
 & \leq \int_{\lbrace \vert u\vert\geq T\rbrace}\widehat{G}_{x}^*(u(x))dx+\int_{\lbrace \vert u\vert\leq \delta\rbrace}\widehat{G}_{x}(u(x))dx\nonumber\\
 & +\max\left\lbrace T^{\ell_{\widehat{A}_x}}, T^{m_{\widehat{A}_x}}\right\rbrace\int_{\lbrace \delta<\vert u\vert<T\rbrace}\widehat{A}_{x}(1)dx\nonumber\\
 & \leq \int_{\lbrace \vert u\vert\geq T\rbrace}\widehat{G}_{x}^*(u(x))dx+\int_{\lbrace \vert u\vert\leq \delta\rbrace}\widehat{G}_{x}(u(x))dx\nonumber\\
 & +C_{10}\max\left\lbrace T^{\ell_{\widehat{A}_x}}, T^{m_{\widehat{A}_x}}\right\rbrace\left\vert\lbrace \delta<\vert u\vert<T\rbrace\right\vert
\end{align}
 for some constant $C_{10}>0$.\\
It's clear that if $\left \vert \lbrace \delta<\vert u\vert<T\rbrace\right\vert<+\infty$, we get our desired result. Then, to the end of the proof, it sufficient to show that $\left \vert \lbrace \delta<\vert u\vert<T\rbrace\right\vert<+\infty$. In fact, we argue by contradiction, suppose that
\begin{equation}\label{2eq90}
    \left \vert \lbrace \delta<\vert u\vert<T\rbrace\right\vert=+\infty.
\end{equation}
 Using \eqref{bf}, and Lemma \ref{lem1}, we obtain
    \begin{align}\label{ed290}
    |\{\delta<|u|<T\}| &\leq \int_{\{\delta<|u|<T\}}\frac{1}{\widehat{G}_x(\delta)}\widehat{G}_x(u(x))dx\nonumber\\
    & \leq \frac{1}{\min\left\lbrace \delta^{g^-},\delta^{g^+} \right\rbrace} \int_{\{\delta<|u|<T\}}\frac{1}{\widehat{G}_x(1)}\widehat{G}_x(u(x))dx\nonumber\\
    & \leq \frac{1}{C_2\min\left\lbrace \delta^{g^-},\delta^{g^+} \right\rbrace} \int_{\{\delta<|u|<T\}}\widehat{G}_x(u(x))dx\nonumber\\
    & \leq \frac{1}{C_2\min\left\lbrace \delta^{g^-},\delta^{g^+} \right\rbrace} \int_{\mathbb{R}^d}\widehat{G}_x(u(x))dx<+\infty.
    \end{align}
  Thus, a contradiction holds with \eqref{2eq90}.
  This completes the proof of  assertion $(2)$.
\end{proof}
\section{Proof of Theorems \ref{comp1} and \ref{compact}}
Before starting, we recall the definition of the  weighted
fractional Musielak-Sobolev space $\mathbb{E}$
$$\mathbb{E}=\bigg{\{}u\in W^{s,G_{x,y}}(\mathbb{R}^{d}):\
\int_{\mathbb{R}^{d}}V(x)\widehat{G}_x(u)dx<\infty\bigg{\}}.$$ The
space $\mathbb{E}$ is equipped with the following norm
$$\|u\|_\mathbb{E}:= [u]_{s,G_{x,y}}+\|u\|_{(\mathbb{E})},$$ where
\begin{equation}\label{2eq19}
\|u\|_{(\mathbb{E})}:=\inf\bigg{\{}\lambda>0:\ \int_{\mathbb{R}^{d}}V(x)\widehat{G}_x\bigg{(}\frac{u(x)}{\lambda}\bigg{)}dx\leq1\bigg{\}}.
\end{equation}
\begin{lemma}\label{lem33} Assume that $(g_{1})-(g_4)$ and $(V_{1})$ are satisfied. Then,
    $$\min\left\lbrace \Vert u\Vert_{(\mathbb{E})}^{g^-},\Vert u\Vert_{(\mathbb{E})}^{g^+}\right\rbrace\leq \displaystyle\int_{\mathbb{R}^{d}}V(x)\widehat{G}_x(u)dx\leq \max\left\lbrace \Vert u\Vert_{(\mathbb{E})}^{g^-},\Vert u\Vert_{(\mathbb{E})}^{g^+}\right\rbrace,\ \ \text{for all}\ u\in \mathbb{E}.$$
\end{lemma}

\begin{proof}
Letting $u\in \mathbb{E}\setminus\lbrace 0\rbrace$, and choosing $\tau=\|u\|_{(\mathbb{E})}$ in Lemma \ref{lem1}, we obtain
   $$
    \widehat{G}_x( u)\leq \max\left\lbrace \Vert u\Vert_{(\mathbb{E})}^{g^-},\Vert u\Vert_{(\mathbb{E})}^{g^+}\right\rbrace \widehat{G}_x\bigg{(}\frac{u}{\|u\|_{(\mathbb{E})}}\bigg{)},\ \text{for all}\ x\in\mathbb{R}^{d}.$$
   Then, from assumption $(V_1)$, we have $$
   V(x)\widehat{G}_x( u)\leq \max\left\lbrace \Vert u\Vert_{(\mathbb{E})}^{g^-},\Vert u\Vert_{(\mathbb{E})}^{g^+}\right\rbrace V(x)\widehat{G}_x\bigg{(}\frac{u}{\|u\|_{(\mathbb{E})}}\bigg{)},\ \text{for all}\ x\in\mathbb{R}^{d}.$$
   It follows, by \eqref{2eq19}, that

   $$\int_{\mathbb{R}^{d}}V(x)\widehat{G}(u)dx\leq \max\left\lbrace \Vert u\Vert_{(\mathbb{E})}^{g^-},\Vert u\Vert_{(\mathbb{E})}^{g^+}\right\rbrace.$$
  Now, letting $\varepsilon>0$, and choosing $\tau=\|u\|_{(\mathbb{E})}-\varepsilon$ in Lemma \ref{lem1}, as above  we get
   $$\min\left\lbrace \left(\Vert u\Vert_{(\mathbb{E})}-\varepsilon\right)^{g^-},\left(\Vert u\Vert_{(\mathbb{E})}-\varepsilon\right)^{g^+}\right\rbrace V(x)\widehat{G}_x\bigg{(}\frac{u}{\|u\|_{(\mathbb{E})}-\epsilon}\bigg{)}\leq V(x)\widehat{G}_x( u),\ \text{for all}\ x\in\mathbb{R}^{d}.$$
   Thus, $$\int_{\mathbb{R}^{d}}V(x)\widehat{G}_x( u)dx\geq\min\left\lbrace \left(\Vert u\Vert_{(\mathbb{E})}-\varepsilon\right)^{g^-},\left(\Vert u\Vert_{(\mathbb{E})}-\varepsilon\right)^{g^+}\right\rbrace .$$
  Passing to the limit as $\varepsilon\rightarrow0$ in the previous inequality, we infer that
$$\int_{\mathbb{R}^{d}}V(x)\widehat{G}_x( u)dx\geq\min\left\lbrace \Vert u\Vert_{(\mathbb{E})}^{g^-},\Vert u\Vert_{(\mathbb{E})}^{g^+}\right\rbrace.$$
   This ends the proof.
\end{proof}

\begin{lemma}\label{Aux1}
 Let $\Omega$ be an open subset of $\mathbb{R}^d$ and $\Phi_{x,y}$ be a generalized N-function satisfying the assumption  \eqref{bf} and
there exist $\ell^-,\ell^+\in (1,+\infty)$ such that
$$1<\ell^-\leq\frac{\widehat{\phi}_x(t)t}{\widehat{\Phi}_x(t)}\leq \ell^+,\ \text{for all}\ x\in \Omega\ \text{and}\ t>0,\ \text{where} \ \widehat{\Phi}_x(t):=\int_{0}^t\widehat{\phi}_x(s)ds.$$  Then, the complementary function $\widetilde{\widehat{\Phi}}_x$ of  $\widehat{\Phi}_x$ verifies:
\begin{enumerate}
    \item[$(1)$] the condition \eqref{bf};
    \item[$(2)$] $\displaystyle{1<\widetilde{\ell}^-\leq\frac{\widetilde{\widehat{\phi}}_x(t)t}{\widetilde{\widehat{\Phi}}_x(t)}\leq \widetilde{\ell}^+,\ \text{for all}\ x\in \Omega\ \text{and}\ t>0,\ \text{where} \ \widetilde{\widehat{\Phi}}_x(t):=\int_{0}^t\widetilde{\widehat{\phi}}_x(s)ds}$ and $\displaystyle{\widetilde{\ell}^-=\frac{\ell^-}{\ell^--1}}$ and $\displaystyle{\widetilde{\ell}^+=\frac{\ell^+}{\ell^+-1}}$.
\end{enumerate}
\end{lemma}
\begin{proof}
  $(1)$  From the definition \eqref{2eq50}, we have
  $$\widetilde{\widehat{\Phi}}_{x}(t)=\widetilde{\Phi}(x,x,t):=\sup_{\tau\geq 0}\left( t\tau-\widehat{\Phi}_{x}(\tau)\right),\ \ \text{for all}\ x\in \Omega\ \text{and all}\ t>0.$$
  Thus,
  \begin{equation}\label{2eq52}
   \widetilde{\widehat{\Phi}}_{x}(1)=\sup_{\tau\geq 0}\left( \tau-\widehat{\Phi}_{x}(\tau)\right),\ \ \text{for all}\ x\in \Omega.
  \end{equation}
  On the other side, by Lemma \ref{lem1}, it yields that
  $$\min \left\lbrace \tau^{\ell^-},\tau^{\ell^+}\right\rbrace \widehat{\Phi}_{x}(1)\leq \widehat{\Phi}_{x}(\tau)\leq \max \left\lbrace \tau^{\ell^-},\tau^{\ell^+}\right\rbrace\widehat{\Phi}_{x}(1),\ \ \text{for all}\ x\in \Omega\ \text{and all}\ \tau>0.$$
  It follows, by \eqref{bf}, that
  $$\min \left\lbrace \tau^{\ell^-},\tau^{\ell^+}\right\rbrace C_1\leq \widehat{\Phi}_{x}(\tau)\leq \max \left\lbrace \tau^{\ell^-},\tau^{\ell^+}\right\rbrace C_2,\ \ \text{for all}\ x\in \Omega\ \text{and all}\ \tau>0.$$
  Therefore,
  $$\tau-\max \left\lbrace \tau^{\ell^-},\tau^{\ell^+}\right\rbrace C_2\leq \tau -\widehat{\Phi}_{x}(\tau)\leq \tau -\min \left\lbrace \tau^{\ell^-},\tau^{\ell^+}\right\rbrace C_1,\ \ \text{for all}\ x\in \Omega\ \text{and all}\ \tau>0.$$
  Hence, in light of \eqref{2eq52}
  \begin{equation}\label{2eq53}
   \sup_{\tau\geq 0}\left(\tau-\max \left\lbrace \tau^{\ell^-},\tau^{\ell^+}\right\rbrace C_2\right)\leq \widetilde{\widehat{\Phi}}_{x}(1)\leq \sup_{\tau\geq 0}\left(\tau -\min \left\lbrace \tau^{\ell^-},\tau^{\ell^+}\right\rbrace C_1\right),\ \ \text{for all}\ x\in \Omega.
  \end{equation}
  According to the fact that $1<\ell^-\leq \ell^+$ and to the previous inequality, we conclude that there are $C_5,C_6>0$ such that
 $$ C_5\leq \widetilde{\widehat{\Phi}}_{x}(1)\leq C_6,\ \ \text{for all}\ x\in \Omega.$$
 This gives $(1)$. For the assertion $(2)$ see \cite{M1}.
\end{proof}
\begin{lemma}\label{lem4}
 Let $\Omega$ be an open subset of $\mathbb{R}^d$ and $B\subset \Omega$ be measurable with
$\vert B\vert \in (0,+\infty)$. Let $\Phi_{x,y}$ be a generalized N-function satisfying the assumption  \eqref{bf} and
there exist $\ell^-,\ell^+\in (1,+\infty)$ such that
$$1<\ell^-\leq\frac{\widehat{\phi}_x(t)t}{\widehat{\Phi}_x(t)}\leq \ell^+,\ \text{for all}\ x\in \Omega\ \text{and}\ t>0,\ \text{where} \ \widehat{\Phi}_x(t):=\int_{0}^t\widehat{\phi}_x(s)ds.$$
 Then, we have $$C_3\min \left\lbrace \vert B\vert^{\frac{1}{\ell^-}},\vert B\vert^{\frac{1}{\ell^+}}\right\rbrace\leq \Vert \chi_{_B}\Vert_{L^{\widehat{\Phi}_{x}}(\Omega)}\leq C_4\max \left\lbrace \vert B\vert^{\frac{1}{\ell^-}},\vert B\vert^{\frac{1}{\ell^+}}\right\rbrace,\ \ \text{for some}\ C_3,C_4>0.$$
\end{lemma}
\begin{proof}
 Exploiting Lemma \ref{lem1}, we infer that
 \begin{equation}\label{2eq30}
 \min \left\lbrace \Vert \chi_{_B}\Vert_{L^{\widehat{\Phi}_{x}}(\Omega)}^{\ell^-},\Vert \chi_{_B}\Vert_{L^{\widehat{\Phi}_{x}}(\Omega)}^{\ell^+}\right\rbrace\leq \int_{\Omega}\widehat{\Phi}_{x}(\chi_{_B})dx\leq \max \left\lbrace \Vert \chi_{_B}\Vert_{L^{\widehat{\Phi}_{x}}(\Omega)}^{\ell^-},\Vert \chi_{_B}\Vert_{L^{\widehat{\Phi}_{x}}(\Omega)}^{\ell^+}\right\rbrace.
 \end{equation}
 On the other side, using assumption \eqref{bf}, we find that
 \begin{equation}\label{2eq31}
    C_1\vert B\vert \leq\int_{\Omega}\widehat{\Phi}_{x}(\chi_{_B})dx=\int_{B}\widehat{\Phi}_{x}(1)dx=\int_{B}\Phi(x,x,1)dx\leq C_2\vert B\vert.
 \end{equation}
 Putting together \eqref{2eq30} and \eqref{2eq31}, we deduce that
 $$\min \left\lbrace \left(C_1\vert B\vert\right)^{\frac{1}{\ell^-}},\left(C_1\vert B\vert\right)^{\frac{1}{\ell^+}}\right\rbrace\leq\Vert \chi_{_B}\Vert_{L^{\widehat{\Phi}_{x}}(\Omega)}\leq \max \left\lbrace \left(C_2\vert B\vert\right)^{\frac{1}{\ell^-}},\left(C_2\vert B\vert\right)^{\frac{1}{\ell^+}}\right\rbrace.$$
 This ends the proof.
\end{proof}
\subsection{Proof of Theorem \ref{comp1} (Compact embedding)}

Let $\lbrace u_{n}\rbrace_{n\in\mathbb{N}}$ be a sequence in $\mathbb{E}$ such that $u_{n}\rightharpoonup u$  in $\mathbb{E}$. Therefore, $u_{n}\rightarrow u$ in $L_{loc}^{\widehat{A}_x}(\mathbb{R}^{d})$, where $\widehat{A}_x$ is a generalized N-function such that $\widehat{G}_x\prec\widehat{A}_x\prec\prec\widehat{G}_x^{*}$, see Theorem \ref{thm2}. Our aim is to prove that $u_{n}\rightarrow u$ in $L^{\widehat{G}_x}(\mathbb{R}^{d})$. Indeed: in light of Br\'ezis-Lieb's theorem \cite{M18}, we just need to show that
 $$\alpha_{n}:=\int_{\mathbb{R}^{d}}\widehat{G}_x(u_{n})dx\rightarrow\int_{\mathbb{R}^{d}}\widehat{G}_x(u)dx.
 $$
It's clear that $\lbrace\alpha_{n}\rbrace_{n\in\mathbb{N}}$ is bounded. Then, up to a subsequence still denoted by $\alpha_n$, we have $\alpha_{n}\rightarrow \alpha$. Hence, it follows from Fatou's Lemma that $$\alpha\geq\int_{\mathbb{R}^d}\widehat{G}_x(u)dx.$$
By using local convergence and the fact that $\widehat{G}_x$ satisfies the $\Delta_{2}$-condition, we get
 \begin{equation}\label{c}
  \int_{B_{r}(0)}\widehat{G}_x(u_{n})dx\rightarrow \int_{B_{r}(0)}\widehat{G}_x(u)dx.
 \end{equation}

\vspace{0,3cm}

\noindent \textbf{Claim:} For each $\varepsilon>0$, there exists $r_\varepsilon>0$ such that
 \begin{equation}\label{c1}
  \int_{B_{r_\varepsilon}^{c}(0)}\widehat{G}_x(u_{n})dx<\varepsilon,\ \ \text{for}\ n\in\mathbb{N}\ \text{large}.
 \end{equation}
Indeed: For given $\varepsilon>0$, let $L>0$ be such that
 \begin{equation}\label{ml1}
  \frac{2}{\varepsilon}\max\left\{B_2,\sup_n\max\left\lbrace \Vert u_n\Vert_{\mathbb{E}}^{g^-},\Vert u_n\Vert_{\mathbb{E}}^{g^+}\right\rbrace\right\}<L,
 \end{equation}
where $B_1$ will be defined below. Using Lemma \ref{Aux}, then we can find a generalized N-function  $\widehat{R}_x$ such that $$\widehat{R}_x\circ \widehat{G}_x\prec\widehat{G}_x^*,\ \ \text{for all}\ x\in \mathbb{R}^d.$$
Which is equivalent to that there exist $C_7>0$ and $T>0$ such that
\begin{equation}\label{2eq54}
\widehat{R}_x(\widehat{G}_x(t))\leq C_7\widehat{G}_x^*(t),\ \ \text{for all}\ x\in \mathbb{R}^d\ \text{and all}\ t\geq T.
\end{equation}
Now, we define the function $f:\mathbb{R}_+\longrightarrow \mathbb{R}$ by
$$f(t):=\max \left\lbrace  t^{\frac{1}{\widetilde{r}^-}},t^{\frac{1}{\widetilde{r}^+}}\right\rbrace, \quad \text{all all}\  t\geq 0.$$
Here $\displaystyle{\widetilde{r}^-:=\frac{r^-}{r^--1}}$ and   $\displaystyle{\widetilde{r}^+:=\frac{r^+}{r^+-1}}$, where  $r^-$ and $r^+$ are defined in Lemma \ref{Aux}.
It's clear that $f$ is a continuous function and $\displaystyle\lim_{t\to 0}f(t)=0$. According to $(V_{1})$, we can choose $r>1$ sufficiently large such that
 \begin{equation}\label{ml2}
  f\left(\vert \{x\in B_{r}^{c}(0):V(x)<L \} \vert\right)\leq \frac{\varepsilon}{2B_{2}}.
 \end{equation}

 Let us define the sets
 \[
  \mathcal{A}:=\{x\in B_{r}^{c}(0):V(x)\geq L \}
 \quad \mbox{and} \quad
 \mathcal{B}:=\{x\in B_{r}^{c}(0):V(x)<L \}.
 \]
In view of Lemma \ref{lem33}, we have
 \begin{equation*}
  \int_{\mathcal{A}}\widehat{G}_x(u_{n})dx \leq  \int_{\mathcal{A}}\frac{V(x)}{L}\widehat{G}_x(u_{n})dx\\
\leq  \frac{1}{L}\max\left\lbrace \Vert u_n\Vert_{(\mathbb{E})}^{g^-},\Vert u_n\Vert_{(\mathbb{E})}^{g^+}\right\rbrace\\
\leq  \frac{1}{L}\max\left\lbrace \Vert u_n\Vert_{\mathbb{E}}^{g^-},\Vert u_n\Vert_{\mathbb{E}}^{g^+}\right\rbrace<\frac{\varepsilon}{2}.
 \end{equation*}
On the other side, it follows from Lemma \ref{Aux} and H\"{o}lder's inequality that
 \begin{eqnarray}\label{ineq1}
  \int_{\mathcal{B}}\widehat{G}_x(u_{n}) dx& \leq & 2\left\|\widehat{G}_x(u_n)\right\|_{L^{\widehat{R}_x}(\mathbb{R}^d)}\left\| \chi_{_\mathcal{B}}\right\|_{L^{\widetilde{\widehat{R}}_x}(\mathbb{R}^d)},
 \end{eqnarray}
where $\widetilde{\widehat{R}}_x$ is the complementary function of $\widehat{R}_x$.\\
Using Lemmas \ref{Aux}, \ref{Aux1} and \ref{lem4}, we see that
\begin{eqnarray}\label{ineq2}
        \| \chi_{_\mathcal{B}}\|_{L^{\widetilde{\widehat{R}}_x}(\mathbb{R}^d)}\leq C_4\max \left\lbrace \vert \mathcal{B}\vert^{\frac{1}{\widetilde{r}^-}},\vert \mathcal{B}\vert^{\frac{1}{\widetilde{r}^+}}\right\rbrace=C_4f(\vert \mathcal{B}\vert).
\end{eqnarray}
Using \eqref{2eq54}, Lemma \ref{lem1} and the fact that $\widehat{G}_x$ and $\widehat{R}_x$ satisfy \eqref{bf}, we find that
\begin{eqnarray}\label{ineq3}
\int_{\mathcal{B}}\widehat{R}_x(\widehat{G}_x(u_n))dx&=&    \int_{\mathcal{B}\cap[|u_n|\leq T]}\widehat{R}_x(\widehat{G}_x(u_n))dx+\int_{\mathcal{B}\cap[|u_n|> T]}\widehat{R}_x(\widehat{G}_x(u_n))dx\nonumber\\
 &\leq& \int_{\mathcal{B}\cap[|u_n|\leq T]}\widehat{R}_x(\widehat{G}_x(T))dx+C_7\sup_n\int_{\mathbb{R}^d} \widehat{G}_x^*(u_n)dx\nonumber\\
&\leq& \int_{\mathcal{B}\cap[|u_n|\leq T]}\widehat{R}_x\left(\max\left\lbrace T^{g^-},T^{g^+}\right\rbrace\widehat{G}_x(1)\right)dx+C_7\sup_n\int_{\mathbb{R}^d} \widehat{G}_x^*(u_n)dx\nonumber\\
&\leq& \int_{\mathcal{B}\cap[|u_n|\leq T]}\widehat{R}_x\left(\max\left\lbrace T^{g^-},T^{g^+}\right\rbrace C_8\right)dx+C_7\sup_n\int_{\mathbb{R}^d} \widehat{G}_x^*(u_n)dx\nonumber\\
&\leq& \int_{\mathcal{B}\cap[|u_n|\leq T]}\max\left\lbrace \left(\max\left\lbrace T^{g^-},T^{g^+}\right\rbrace C_8\right)^{r^-}, \left(\max\left\lbrace T^{g^-},T^{g^+}\right\rbrace C_8\right)^{r^+}\right\rbrace\widehat{R}_x\left(1\right)dx\nonumber\\ &+&C_7\sup_n\int_{\mathbb{R}^d} \widehat{G}_x^*(u_n)dx\nonumber\\
&\leq& \int_{\mathcal{B}\cap[|u_n|\leq T]}C_9\max\left\lbrace \left(\max\left\lbrace T^{g^-},T^{g^+}\right\rbrace C_8\right)^{r^-}, \left(\max\left\lbrace T^{g^-},T^{g^+}\right\rbrace C_8\right)^{r^+}\right\rbrace dx\nonumber\\ &+&C_7\sup_n\int_{\mathbb{R}^d} \widehat{G}_x^*(u_n)dx\nonumber\\
& \leq & \vert \mathcal{B}_1\vert T_1 + C_7\sup_n\int_{\mathbb{R}^d} \widehat{G}_x^*(u_n)dx
\end{eqnarray}
where $\mathcal{B}\subset \mathcal{B}_1$ for all $0<\varepsilon<1$ small enough and $$\displaystyle{T_1}:=C_9\max\left\lbrace \left(\max\left\lbrace T^{g^-},T^{g^+}\right\rbrace C_8\right)^{r^-}, \left(\max\left\lbrace T^{g^-},T^{g^+}\right\rbrace C_8\right)^{r^+}\right\rbrace.$$ Next, we define
$$B_2:=2C_0C_4|\mathcal{B}_1|T_1+2C_0C_4\sup_n\int_{\mathbb{R}^d} \widehat{G}_x^*(u_n)dx.$$
Thus, by \eqref{ml2}, \eqref{ineq1}, \eqref{ineq2}, and \eqref{ineq3}, we obtain
$$
\int_{\mathcal{B}}\widehat{G}_x(u_{n})dx\leq B_2 \max \left\lbrace \vert \mathcal{B}\vert^{\frac{1}{\widetilde{r}^-}},\vert \mathcal{B}\vert^{\frac{1}{\widetilde{r}^+}}\right\rbrace=B_2 f(|\mathcal{B}|)\leq \frac{\varepsilon}{2}.
$$
Therefore,
 \[
  \int_{B_{r}^{c}(0)}\widehat{G}_x(u_{n})dx=\int_{\mathcal{A}}\widehat{G}_x(u_{n})dx+\int_{\mathcal{B}}\widehat{G}_x(u_{n})dx<\varepsilon.
 \]
Thus, the proof of Claim.\\
Exploiting the Claim, we find that
 \begin{eqnarray*} \int_{\mathbb{R}^{d}}\widehat{G}_x(u) dx& = & \int_{B_{r}(0)}\widehat{G}_x(u)dx+\int_{B_{r}^{c}(0)}\widehat{G}_x(u)dx\\
 & \geq & \lim_{n\rightarrow\infty}\int_{B_{r}(0)}\widehat{G}_x(u_{n})dx\\
 & = & \lim_{n\rightarrow\infty}\int_{\mathbb{R}^{d}}\widehat{G}_x(u_{n})dx-\lim_{n\rightarrow\infty}\int_{B_{r}^{c}(0)}\widehat{G}_x(u_{n})dx\\
 & \geq & \alpha-\varepsilon.
 \end{eqnarray*}
This ends the proof.
\subsection{Proof of Theorem \ref{compact} (Compact embedding)}

  Since $\widehat{A}_x\prec \prec\widehat{G}_x^{*}$, for given $\varepsilon>0$, there exists $T>0$ such that
     \begin{equation}\label{ml4}
      \frac{\widehat{A}_x(\vert t\vert)}{\widehat{G}_x^{*}(\vert t\vert)}\leq \frac{\varepsilon}{2\kappa}, \quad \vert t\vert \geq T,\ \ \text{for all}\ x\in \mathbb{R}^d
     \end{equation}
    where $\kappa>0$ will be chosen later. Let $\lbrace u_{n}\rbrace_{n\in \mathbb{N} }\subset \mathbb{E}$ be a sequence such that $u_{n}\rightharpoonup 0$  in $\mathbb{E}$. In view of Theorem \ref{comp1}, it follows that
 \begin{equation}\label{2eq20}
     u_{n}\rightarrow0\  \ \text{in}\ \  L^{\widehat{G}_x}(\mathbb{R}^{d}).
 \end{equation}
 Next, we consider the following decomposition:
     \begin{equation}\label{ml5}
      \int_{\mathbb{R}^{d}}\widehat{A}_x(\vert u_{n}\vert)dx=\int_{\{\vert u_{n}\vert\geq T\}}\widehat{A}_x(\vert u_{n}\vert)dx+\int_{\{\vert u_{n}\vert<T\}}\widehat{A}_x(\vert u_{n}\vert)dx.
     \end{equation}
    Using Theorem \ref{thm1}, we can we choose
      \begin{equation}\label{k}    \kappa:=\sup_n\int_{\mathbb{R}^{d}}\widehat{G}_x^*(\vert u_{n}\vert)dx<+\infty.
      \end{equation}
     It follows, from \eqref{ml4}, that
      \begin{equation}\label{ml6}
       \int_{\{\vert u_{n}\vert\geq T\}}\widehat{A}_x(\vert u_{n}\vert)dx\leq \frac{\varepsilon}{2\kappa}\int_{\mathbb{R}^{d}}\widehat{G}_x^*(\vert u_{n}\vert)dx\leq \frac{\varepsilon}{2}.
      \end{equation}
To the end of the proof, we shall study the integral in \eqref{ml5} on the set  $\left\lbrace\vert u_{n}\vert< T\right\rbrace$. For that, we use one of the assumptions \eqref{mla1} or \eqref{mla2}.  

      \subsubsection{\textbf{Proof of Theorem \ref{compact} assuming \eqref{mla1}}}

      Let $p\in(0,1)$. Using H\"{o}lder's inequality, we find that
      \begin{equation}\label{ml9}
       \int_{\{\vert u_{n}\vert<T\}}\widehat{A}_x(\vert u_{n}\vert)dx\leq \left[\int_{\{\vert u_{n}\vert<T\}}\left(\frac{\widehat{A}_x(\vert u_{n}\vert)}{\widehat{G}_x(\vert u_{n}\vert)^{p}} \right)^\frac{1}{1-p}dx \right]^{1-p}\left[\int_{\mathbb{R}^{d}}\widehat{G}_x(\vert u_{n}\vert)dx \right]^{p}.
      \end{equation}
     \vspace{0,3cm}

    In view of assumption \eqref{mla1}, there exist $\delta,C>0$ such that $$\widehat{A}_x(\vert u_n\vert)\leq C\widehat{G}_x(\vert u_n\vert),\ \text{for all}\ \vert u_n\vert \leq \delta\ \text{and}\ x\in \mathbb{R}^d.$$
 If $\delta<T$, then, from \eqref{bf} and Lemma \ref{lem1}, we find that
       \[
        \frac{\widehat{A}_x(\vert u_n\vert)}{\widehat{G}_x(\vert u_n\vert)}\leq \frac{\widehat{A}_x(T)}{\widehat{G}_x(\delta)}\leq \frac{\max\left\lbrace T^{\ell_{\widehat{A}_x}},T^{m_{\widehat{A}_x}} \right\rbrace \widehat{A}_x(1)}{\min\left\lbrace \delta^{g^-},\delta^{g^+} \right\rbrace \widehat{G}_x(1)}\leq \widetilde{C}\frac{\max\left\lbrace T^{\ell_{\widehat{A}_x}},T^{m_{\widehat{A}_x}} \right\rbrace }{\min\left\lbrace \delta^{g^-},\delta^{g^+} \right\rbrace },\ \text{for all}\ \vert u_n\vert\in[\delta,T]\ \text{and}\  x\in \mathbb{R}^d
       \]
    for some constant $\widetilde{C}>0$ independent from $x$.
Therefore,
       \[
        \frac{\widehat{A}_x(\vert u_n\vert)}{\widehat{G}_x(\vert u_n\vert)}\leq \widetilde{C}_1^{1-p},\ \text{for all}\ \vert u_n\vert\leq T \ \text{and}\  x\in \mathbb{R}^d,
       \]
    with $\widetilde{C}_1^{1-p}:=\max\left\{C, \widetilde{C}\frac{\max\left\lbrace T^{\ell_{\widehat{A}_x}},T^{m_{\widehat{A}_x}} \right\rbrace }{\min\left\lbrace \delta^{g^-},\delta^{g^+} \right\rbrace }\right\}$. Thus,
     \begin{equation}\label{ml7}
       \left(\frac{\widehat{A}_x(\vert u_n\vert)}{\widehat{G}_x(\vert u_n\vert)^{p}} \right)^\frac{1}{1-p}\leq \widetilde{C}_1\widehat{G}_x(\vert u_n\vert),\ \text{for all}\ \vert u_n\vert\leq T \ \text{and}\  x\in \mathbb{R}^d.
      \end{equation}
Finally, from \eqref{2eq20},  there exists $n_{0}\in\mathbb{N}$ such that
       \begin{equation}\label{ml8}
        \int_{\mathbb{R}^{d}}\widehat{G}_x(\vert u_{n}\vert)dx<\frac{\varepsilon}{2\widetilde{C}_1^{1-p}} , \quad \mbox{for all } n>n_{0}.
       \end{equation}
Hence, by \eqref{ml6}--\eqref{ml8}, we deduce that
       \begin{equation}\label{mla8}
        \int_{\mathbb{R}^{d}}\widehat{A}_x(\vert u_{n}\vert)dx<\varepsilon.
       \end{equation}
       This completes the proof.
    \subsubsection{\textbf{Proof of Theorem \ref{compact} assuming \eqref{mla2}}}
Now, we shall suppose that \eqref{mla2} holds. In this
case, there exists $n_0\in \mathbb{N}$ such that
\begin{equation}\label{2eq55}
\int_{\mathbb{R}^d} \widehat{G}_x(\vert u_n\vert) dx\leq \min\left\lbrace \frac{\varepsilon}{4\widetilde{C}_3},\left(\frac{\varepsilon}{4k^{1-a}}\right)^{\frac{1}{a}} \right\rbrace,\ \ \text{for all}\ n\geq n_0
\end{equation}
where $\widetilde{C}_3$ will be defined later.\\
In light of assumption \eqref{mla2}, we see that
\begin{align}\label{2eq56}
    \int_{\left\lbrace \vert u_n\vert\leq1\right\rbrace}\widehat{A}_x(\vert u_n\vert)dx& \leq \int_{\left\lbrace \vert u_n\vert\leq1\right\rbrace}\widehat{G}_x(\vert u_n\vert)^a\widehat{G}_x^*(\vert u_n\vert)^{1-a}dx\nonumber\\
    & \leq \left(\int_{\left\lbrace \vert u_n\vert\leq1\right\rbrace}\widehat{G}_x(\vert u_n\vert)dx\right)^a\left(\int_{\left\lbrace \vert u_n\vert\leq1\right\rbrace}\widehat{G}_x^*(\vert u_n\vert)dx\right)^{1-a}\nonumber\\
   & \leq\frac{\varepsilon}{4}.
\end{align}
If $1<T$, it follows, from \eqref{2eq55} , \eqref{bf} and Lemma \ref{lem1}, that for all  $n\geq n_0$
 we have
 \begin{align}\label{2eq57}
    \int_{\left\lbrace 1\leq \vert u_n\vert\leq T\right\rbrace}\widehat{A}_x(\vert u_n\vert)dx&\leq \int_{\mathbb{R}^d}\frac{\widehat{A}_x(T)}{\widehat{G}_x(1)}\widehat{G}_x(\vert u_n\vert)dx\nonumber\\
    & \leq \max\left \lbrace T^{\ell_{\widehat{A}_x}},T^{m_{\widehat{A}_x}} \right\rbrace \int_{\mathbb{R}^d}\frac{\widehat{A}_x(1)}{\widehat{G}_x(1)}\widehat{G}_x(\vert u_n\vert)dx\nonumber\\
    &\leq \widetilde{C}_2\max\left \lbrace T^{\ell_{\widehat{A}_x}},T^{m_{\widehat{A}_x}} \right\rbrace \int_{\mathbb{R}^d}\widehat{G}_x(\vert u_n\vert)dx\nonumber\\
    &\leq \widetilde{C}_3\int_{\mathbb{R}^d}\widehat{G}_x(\vert u_n\vert)dx\nonumber\\
    &<\frac{\varepsilon}{4},
\end{align}
for some constants $\widetilde{C}_2>0$ and $\displaystyle{\widetilde{C}_3:=\widetilde{C}_2\max\left \lbrace T^{\ell_{\widehat{A}_x}},T^{m_{\widehat{A}_x}} \right\rbrace }$.\\
By \eqref{2eq55}, \eqref{2eq56} and \eqref{2eq57}, we conclude \eqref{mla8}.

 \section{Proof of Theorem \ref{lions} (Lions Lemma type result)}

    Let $\lbrace u_{n}\rbrace_{n\in\mathbb{N}}\subset \mathbb{E}$ be satisfying \eqref{lionss}.  Since $\widehat{A}_x\prec\prec \widehat{G}_x^{*}$, for given $\varepsilon>0$, there exists $T>0$ such that
    \begin{equation}\label{l1}
    \frac{\widehat{A}_x(\vert t\vert)}{\widehat{G}_x^{*}(\vert t\vert)}\leq \frac{\varepsilon}{3\kappa}, \quad \text{for all}\ \vert t\vert \geq T\ \text{and}\ x\in \mathbb{R}^d,
    \end{equation}
    where $\kappa$ is defined in \eqref{k}
     \[
     \kappa=\sup_n\int_{\mathbb{R}^{d}}\widehat{G}_x^{*}(\vert u_{n}\vert)dx<+\infty.
     \]
    From \eqref{mla1b}, there exists $\delta>0$ such that
    \begin{eqnarray}\label{l3}
    \frac{\widehat{A}_x(|t|)}{\widehat{G}_x(|t|)}\leq \frac{\varepsilon}{3\theta},\qquad \text{for all}\ |t|<\delta\ \text{and}\ x\in \mathbb{R}^d,
    \end{eqnarray}
    where $$\theta:=\sup_{n}\int_{\mathbb{R}^d}\widehat{G}_x(|u_n|)dx<+\infty\ (\text{from Theorem \ref{thm1}}).$$
    Let us consider the following decomposition
    \begin{equation}\label{l2}
     \int_{\mathbb{R}^{d}}\widehat{A}_x(\vert u_{n}\vert)dx=\int_{\{\vert u_{n}\vert\leq \delta\}}\widehat{A}_x(\vert u_{n}\vert)dx+\int_{\{\delta<|u_n|<T\}}\widehat{A}_x(\vert u_{n}\vert)dx+\int_{\{\vert u_{n}\vert\geq T\}}\widehat{A}_x(\vert u_{n}\vert)dx.
    \end{equation}
    In view of \eqref{l1} we have
    \begin{equation}\label{mlb0}
    \int_{\{\vert u_{n}\vert\geq T\}}\widehat{A}_x(\vert u_{n}\vert)dx\leq \frac{\varepsilon}{3\kappa}\int_{\mathbb{R}^{d}}\widehat{G}_x^{*}(\vert u_{n}\vert)dx\leq \frac{\varepsilon}{3}.
    \end{equation}
    It follows, from \eqref{l3}, that
    \begin{eqnarray}\label{mlb1}
        \int_{\{\vert u_{n}\vert\leq\delta\}}\widehat{A}_x(\vert u_{n}\vert)dx \leq \frac{\varepsilon}{3\theta}\int_{\{|u_n|\leq\delta\}}\widehat{A}_x(|u_n|)dx\leq \frac{\varepsilon}{3}.
    \end{eqnarray}
    At this level, there are two cases to consider. In the first one, we suppose that
    \begin{equation}\label{ed1}
    \lim_{n \to \infty} |\{\delta<|u_n|<T\}|= 0.
    \end{equation}
    Thus, there exists $n_{0}\in\mathbb{N}$ such that
     \begin{equation}\label{ml10}
      |\{\delta<|u_n|<T\}|<\dfrac{\varepsilon}{3\widetilde{C}_5\widetilde{C}_4}, \quad n\geq n_{0},
     \end{equation}
   where $\widetilde{C}_4,\widetilde{C}_5>0$ will be defined latter. Hence, we obtain that
    \begin{align}\label{ed2}
    |\{\delta<|u_n|<T\}| &\leq \int_{\{\delta<|u_n|<T\}}\frac{1}{\widehat{G}_x(\delta)}\widehat{G}_x(u_n)dx\nonumber\\
    & \leq \frac{1}{\min\left\lbrace \delta^{g^-},\delta^{g^+} \right\rbrace} \int_{\{\delta<|u_n|<T\}}\frac{1}{\widehat{G}_x(1)}\widehat{G}_x(u_n)dx\nonumber\\
    & \leq \frac{\max\left\lbrace T^{g^-},T^{g^+} \right\rbrace}{\min\left\lbrace \delta^{g^-},\delta^{g^+} \right\rbrace} \int_{\{\delta<|u_n|<T\}}\frac{\widehat{G}_x(1)}{\widehat{G}_x(1)}dx\nonumber\\
   & = \frac{\max\left\lbrace T^{g^-},T^{g^+} \right\rbrace}{\min\left\lbrace \delta^{g^-},\delta^{g^+} \right\rbrace}  |\{\delta<|u_n|<T\}|\nonumber\\
   &=\widetilde{C}_5 |\{\delta<|u_n|<T\}|.
    \end{align}
    For $n\geq n_{0}$, it follows, from Lemma \ref{lem1}, \eqref{bf}, and \eqref{ml10}, that
     \begin{align}\label{ml12}
     \int_{\{\delta<|u_n|<T\}}\widehat{A}_x(u_{n})dx&\leq \int_{\{\delta<|u_n|<T\}}\frac{\widehat{A}_x(T)}{\widehat{G}_x(\delta)}\widehat{G}_x(u_{n})dx\nonumber\\
   &\leq\frac{\max\left\lbrace T^{\ell_{\widehat{A}_x}},T^{m_{\widehat{A}_x}} \right\rbrace }{\min\left\lbrace \delta^{g^-},\delta^{g^+} \right\rbrace }\int_{\{\delta<|u_n|<T\}}\frac{\widehat{A}_x(1)}{\widehat{G}_x(1)}\widehat{G}_x(u_{n})dx\nonumber\\
   & \leq\widetilde{C}_3\frac{\max\left\lbrace T^{\ell_{\widehat{A}_x}},T^{m_{\widehat{A}_x}} \right\rbrace }{\min\left\lbrace \delta^{g^-},\delta^{g^+} \right\rbrace }\int_{\{\delta<|u_n|<T\}}\widehat{G}_x(u_{n})dx\nonumber\\
    & \leq\widetilde{C}_4\int_{\{\delta<|u_n|<T\}}\widehat{G}_x(u_{n})dx\nonumber\\
  & <\dfrac{\varepsilon}{3},
     \end{align}
   for some $\widetilde{C}_3>0$ and $\displaystyle{\widetilde{C}_4:=\widetilde{C}_3\frac{\max\left\lbrace T^{\ell_{\widehat{A}_x}},T^{m_{\widehat{A}_x}} \right\rbrace }{\min\left\lbrace \delta^{g^-},\delta^{g^+} \right\rbrace }}.$\\
    Therefore, by using \eqref{mlb0}, \eqref{mlb1} and \eqref{ml12}, we deduce that
    \begin{equation}
        \int_{\mathbb{R}^d} \widehat{A}_x(u_n)dx \leq \varepsilon,\ \ \text{for each}\ \varepsilon > 0.
    \end{equation}
  This finishes the proof for the first case.

 In the second case, up to a subsequence, we assume that
    \begin{equation}
    \lim_{n \to \infty} |\{\delta<|u_n|<T\}| = M \in (0, \infty).
    \end{equation}
   Let us prove that this case does not hold. For this purpose, we prove the following claim:\\
{\bf Claim:} There exist $y_0\in\mathbb{R}^d$ and $\sigma > 0$ such that \begin{equation}
    0 < \sigma \leq |\{\delta<|u_{n}|<T\}\cap B_r(y_0)|
    \end{equation}
    holds true for a subsequence of $\lbrace u_n\rbrace_{n\in \mathbb{N}}$ which is also labeled as $u_n$. The proof follows arguing by contradiction. Indeed, for each $\varepsilon > 0, k \in \mathbb{N}$ we obtain that
    \begin{equation}\label{ed7}
    |\{\delta<|u_{n}|<T\}\cap B_r(y)| < \frac{\varepsilon}{2^k}
    \end{equation}
    holds for all $y \in \mathbb{R}^d$. Notice also that the last estimate holds for any subsequence of $u_n$. Without loss of generality we take just the sequence $u_n$. Now, choose $\lbrace y_k\rbrace_{k\in \mathbb{N}} \subset \mathbb{R}^d$ such that $\displaystyle{ \cup_{k=1}^{\infty} B_r (y_k) = \mathbb{R}^d}$ and using \eqref{ed7}, we write
    \begin{eqnarray}
    |\{\delta<|u_n|<T\}| &=& |\{\delta<|u_n|<T\} \cap (\cup_{k=1}^{\infty} B_r (y_k))|  \nonumber \\
    &\leq&
    \sum_{k=1}^{\infty}     |\{\delta<|u_n|<T\} \cap B_r(y_k)| \leq
    \sum_{k=1}^{\infty}     \frac{\varepsilon}{2^k} = \varepsilon
    \end{eqnarray}
    where $\varepsilon > 0$ is arbitrary. Up to a subsequence it follows from the last estimate that
    \begin{equation}
    0 < M  = \lim_{n \to \infty} |\{\delta<|u_n|<T\}| \leq \varepsilon
    \end{equation}
    which does not make sense for $\varepsilon \in (0, M)$. Thus the proof of Claim follows.

    At this stage, by using Claim and \eqref{lionss}, \eqref{bf} and Lemma \ref{lem1}, we observe that
    \begin{eqnarray}
    0 &<& \sigma \leq |\{\delta<|u_{n}|<T\}\cap B_r(y_0)| \leq  \int_{B_r(y_0)} \frac{1}{\widehat{G}_x(\delta)}\widehat{G}_x(u_n)dx \nonumber \\
  & \leq& \min\left\lbrace \delta^{g^-},\delta^{g^+} \right\rbrace \int_{B_r(y_0)} \frac{1}{\widehat{G}_x(1)}\widehat{G}_x(u_n)dx \nonumber \\
    &\leq& \widetilde{C}_6\min\left\lbrace \delta^{g^-},\delta^{g^+} \right\rbrace\sup_{y \in \mathbb{R}^d} \int_{B_r(y)} \widehat{G}_x(u_n)dx  \to 0\ \text{as}\ n \to \infty.
    \end{eqnarray}
 This contradiction proves that the second case is impossible. In other words, we prove that $M = 0$ is always verified. Hence, our result follows from the first  case. This ends the proof.

\section{Proof of Theorem \ref{thms} (Strauss radial embedding)}
Let $\lbrace u_n\rbrace_{n\in \mathbb{N}}\subset W^{s,G_{x,y}}_{rad}(\mathbb{R}^d)$ be a bounded sequence. Since $ W^{s,G_{x,y}}_{rad}(\mathbb{R}^d)$ is a reflexive space, up to subsequence, still denoted
by $u_n$,
\begin{equation}\label{2eq62}
   u_n\rightharpoonup 0\ \ \text{in}\ W^{s,G_{x,y}}_{rad}(\mathbb{R}^d).
\end{equation}
Using the continuous embedding $W^{s,G_{x,y}}(\mathbb{R}^d)\hookrightarrow L^{\widehat{G}_x}(\mathbb{R}^d)$, we could find a constant $C > 0$ such that
\begin{equation}\label{2eq63}
    \int_{\mathbb{R}^d}\widehat{G}_x(u_n)dx<C.
\end{equation}
Let us fix $r>0$. Since $u_n$  is radially symmetric for all $n\in \mathbb{N}$,
\begin{equation}\label{2eq64}
  \int_{B_{r}(y_1)}\widehat{G}_x(u_n)dx = \int_{B_{r}(y_2)}\widehat{G}_x(u_n)dx,\ \text{for all}\ y_1,y_2\in  \mathbb{R}^d\ \text{and}\ \vert y_1\vert=\vert y_2\vert.
\end{equation}
 In the sequel, for each $y\in \mathbb{R}^d$, $\vert y\vert >r$, we denote by $\gamma(y)$ the maximum of the integers $j \geq 1$ such that there exist $y_1,y_2,\cdots,y_j\in \mathbb{R}^d$,
 with
 $$\vert y_1\vert=\vert y_2\vert=\cdots=\vert y_j\vert=\vert y\vert\ \ \text{and}\ \ B_r(y_i)\cap B_r(y_k)=\emptyset,\ \ \text{whenever}\ i\neq k.$$
 From the above definition, it is clear that
\begin{equation}\label{2eq65}
    \gamma(y)\longrightarrow +\infty\ \ \text{as}\ \vert y\vert\longrightarrow +\infty.
\end{equation}
 Let $y\in \mathbb{R}^d$, $\vert y\vert>r$ and choose $y_1,\cdots,y_{\gamma(y)}\in \mathbb{R}^d$ as above. Thus, by \eqref{2eq63}, \eqref{2eq64} and \eqref{2eq65}, we obtain
\begin{align*}   C&>\int_{\mathbb{R}^d}\widehat{G}_x(u_n)dx\geq \sum_{i=1}^{\gamma(y) }\int_{B_r(y_i)}\widehat{G}_x(u_n)dx\\
    & \geq \gamma(y) \int_{B_r(y)}\widehat{G}_x(u_n)dx.
\end{align*}
 It follows, by \eqref{2eq64}, that
\begin{equation}\label{2eq66}
  \int_{B_r(y)}\widehat{G}_x(u_n)dx\leq \frac{C}{\gamma(y)} \longrightarrow 0 \ \ \text{as}\ \vert y\vert\longrightarrow +\infty.
\end{equation}
Therefore,  for arbitrary $\varepsilon > 0$, there exists $R_\varepsilon>0$ such that
\begin{equation}\label{2eq67}
 \sup_{\vert y\vert\geq R_\varepsilon} \int_{B_r(y)}  \widehat{G}_x(u_n)dx\leq \varepsilon,\ n\in \mathbb{N}.
\end{equation}
On the other side, by Theorem \ref{thm2}, we have the following compact embedding
$$W^{s,G_{x,y}}\left(B_{r+R_\varepsilon}(0)\right)\hookrightarrow L^{\widehat{G}_x}\left(B_{r+R_\varepsilon}(0)\right).$$
Hence, $u_n\longrightarrow 0$ in $L^{\widehat{G}_x}\left(B_{r+R_\varepsilon}(0)\right)$ which implies that
\begin{equation*}\label{2eq68}
  \int_{B_{r+R_\varepsilon}(0)}  \widehat{G}_x(u_n)dx \longrightarrow 0\ \ \text{as}\ \ n\longrightarrow +\infty.
\end{equation*}
Thus,
\begin{equation}\label{2eq69}
   \sup_{\vert y\vert< R_\varepsilon} \int_{B_{r}(y)}  \widehat{G}_x(u_n)dx \longrightarrow 0\ \ \text{as}\ \ n\longrightarrow +\infty.
\end{equation}
Putting together \eqref{2eq67} and \eqref{2eq69}, and applying Theorem \ref{lions}, we deduce that
$$u_n \longrightarrow 0\ \ \text{in}\ \ L^{\widehat{A}_x}(\mathbb{R}^d).$$
This ends the proof.
\section{Application: Proof of Theorem \ref{thmv}}
\begin{dfn}
    We say that a function $u\in \mathbb{E}$ is a weak solution of problem \eqref{PV} if it verifies
    \begin{equation}\label{2eq70} \int_{\mathbb{R}^d}\int_{\mathbb{R}^d}g_{x,y}\left( \frac{u(x)-u(y)}{\vert x-y\vert^{s}}\right) \frac{v(x)-v(y)}{\vert x-y\vert^{d+s}}dxdy +\int_{\mathbb{R}^d}V(x)g_{x,x}(u)vdx=\int_{\mathbb{R}^d}b(x)\vert u\vert^{p(x)-2}uvdx,
    \end{equation}
    for all $v\in \mathbb{E}$.
\end{dfn}

In view of assumptions $(g_1)-(g_3)$, the functional $I:\mathbb{E}\longrightarrow \mathbb{R}$ given by
 \begin{equation}\label{2eq71}
    I(u):=J_{G_{x,y}}(u)+\int_{\mathbb{R}^d}V(x)\widehat{G}_x(u)dx-\int_{\mathbb{R}^d}\frac{1}{p(x)} b(x)\vert u\vert^{p(x)}dx
 \end{equation}
 is well-defined. Moreover, $I\in C^{1}\left(\mathbb{E},\mathbb{R}\right)$ with the following derivative
 \begin{align}\label{2eq72}
    \langle I^{'}(u),v\rangle_{V}&:= \int_{\mathbb{R}^d}\int_{\mathbb{R}^d}g_{x,y}\left( \frac{u(x)-u(y)}{\vert x-y\vert^{s}}\right) \frac{v(x)-v(y)}{\vert x-y\vert^{d+s}}dxdy+\int_{\mathbb{R}^d}V(x)g_{x,x}(u)vdx\nonumber\\ &-\int_{\mathbb{R}^d}b(x)\vert u\vert^{p(x)-2}uvdx\nonumber\\
     &= \langle J^{'}_{s,G_{x,y}}(u),v\rangle+\int_{\mathbb{R}^d}V(x)g_{x,x}(u)vdx-\int_{\mathbb{R}^d}b(x)\vert u\vert^{p(x)-2}uvdx,\ \ \text{for all}\ u,v\in \mathbb{E}
 \end{align}
 where $\langle \cdot,\cdot\rangle_V$ is the duality brackets for the pair $\left( \mathbb{E}^*,\mathbb{E}\right)$.\\
  From \eqref{2eq71} and \eqref{2eq72}, it's clear that the weak solutions of problem \eqref{PV} are the critical points of the functional $I$.\\

  Now, we recall the following technical lemma which will be useful in the sequel.
\begin{lemma}[Lemma A.1, \cite{M61}]\label{lempx}
   Assume that $h_1 \in L^{\infty}(\mathbb{R}^d)$ such that $h_1 \geq 0$ and $h_1\not\equiv 0$,  a.a. in $\mathbb{R}^d$. Let $h_2:\mathbb{R}^d\longrightarrow \mathbb{R}$ be a bounded and measurable function such that $h_1h_2 \geq 1$, a.a. in $\mathbb{R}^d$. Then for any $ u\in L^{h_1(\cdot)h_2(\cdot)}(\mathbb{R}^d)$,
   $$\Vert\vert u\vert^{h_1(\cdot)} \Vert_{h_2(\cdot)}\leq \Vert u\Vert^{h_1^-}_{h_1(\cdot)h_2(\cdot)}+\Vert u\Vert^{h_1^+}_{h_1(\cdot)h_2(\cdot)}$$
   where $\displaystyle{h_1^-:=\inf_{x\in \mathbb{R}^d}h_1(x), \ h_1^+:=\sup_{x\in \mathbb{R}^d}h_1(x)}$ and $\Vert u\Vert_{h_2(\cdot)}:=\Vert u\Vert_{L^{\widehat{B}_x}(\mathbb{R}^d)}$, with $\widehat{B}_x(t)=\frac{1}{h_2(x)}\vert t\vert^{h_2(x)}.$
\end{lemma}
  \begin{rem}\label{rempx}
   We would like to mention that the function $\widehat{B}_x$ defined in the above lemma is a generalized N-function that satisfies and  \eqref{bf} and  the $\Delta_2$-condition.
  \end{rem}
  \begin{lemma}\label{lemv1}
     Assume that the assumptions \eqref{B}, $(g_1)-(g_5)$ and $(V_1)-(V_2)$ hold. Then, the functional $I$ is coercive.
  \end{lemma}
  \begin{proof}
    Let $u\in \mathbb{E}$. Using Lemmas \ref{lem2}, \ref{lem33}, \ref{lempx}, Theorem \ref{compact} and condition \eqref{B}, we find
   \begin{align*}
    I(u)&=J_{G_{x,y}}(u)+\int_{\mathbb{R}^d}V(x)\widehat{G}_x(u)dx-\int_{\mathbb{R}^d}\frac{1}{p(x)} b(x)\vert u\vert^{p(x)}dx\\
    & \geq \max\left\lbrace\Vert u\Vert_{\mathbb{E}}^{g^-},\Vert u\Vert_{\mathbb{E}}^{g^+}\right\rbrace-\frac{1}{p^-}\Vert b\Vert_{\delta^{'}(\cdot)}\Vert \vert u\vert^{p(\cdot)}\Vert_{\delta(\cdot)}\\
    & \geq \max\left\lbrace\Vert u\Vert_{\mathbb{E}}^{g^-},\Vert u\Vert_{\mathbb{E}}^{g^+}\right\rbrace-\frac{\Vert b\Vert_{\delta^{'}(\cdot)}}{p^-}\left(\Vert u\Vert^{p^-}_{p(\cdot)\delta(\cdot)}+\Vert u\Vert^{p^+}_{p(\cdot)\delta(\cdot)}\right)\\
    & \geq \max\left\lbrace\Vert u\Vert_{\mathbb{E}}^{g^-},\Vert u\Vert_{\mathbb{E}}^{g^+}\right\rbrace-\frac{C\Vert b\Vert_{\delta^{'}(\cdot)}}{p^-}\left(\Vert u\Vert^{p^-}_{\mathbb{E}}+\Vert u\Vert^{p^+}_{\mathbb{E}}\right).
   \end{align*}
   It follows, since $p^-\leq p^+< g^+$, 
  $$ I(u_n)\longrightarrow +\infty\ \  \text{when}\ \ \Vert u_n\Vert_{\mathbb{E}}\longrightarrow +\infty.$$
   This ends the proof.
  \end{proof}
  \begin{lemma}\label{lemps}
   Assume that the assumptions, \eqref{B}, $(g_1)-(g_5)$ and $(V_1)-(V_2)$ hold. Let  $\lbrace u_n\rbrace_{n\in \mathbb{N}}$ be a sequence in $\mathbb{E}$ such that
   \begin{equation}\label{2eq80}
       I(u_n)\longrightarrow c\ \ \text{and}\ I^{'}(u_n)\longrightarrow 0.
   \end{equation}
   Then, up to subsequence, $u_n$ converge in $\mathbb{E}$.
  \end{lemma}
  \begin{proof}
    Let  $\lbrace u_n\rbrace_{n\in \mathbb{N}}$ be a sequence in $\mathbb{E}$ verifying \eqref{2eq80}. It follows, by Lemma \ref{lemv1}, that $u_n$ is bounded in $\mathbb{E}$. Since $\mathbb{E}$ is a reflexive space, up to subsequence still denoted by $u_n$, there exists $u\in \mathbb{E}$ such that
    \begin{equation}\label{2eq81}
        u_n \rightharpoonup u\ \text{in}\ \mathbb{E}.
    \end{equation}
    Therefore, it remains to show that $u_n\longrightarrow u$ in $\mathbb{E}$.\\
 In light of \eqref{2eq80} and \eqref{2eq81}, we can see that
    \begin{equation}\label{2eq82}
     \langle I^{'}(u_n)-I^{'}(u),u_n-u\rangle_V=o_n(1).
    \end{equation}
    On the other side, using $(V_1)$, \eqref{2eq72}, Lemmas \ref{lem1}, \ref{lem55}, \ref{lempx}, Remark \ref{rempx}, and H\"older inequality, we find that
  \begin{align}\label{2eq83}
  \langle I^{'}(u_n)-I^{'}(u),u_n-u\rangle_V&=\langle J^{'}_{s,G_{x,y}}(u_n)-J^{'}_{s,G_{x,y}}(u),u_n-u\rangle+\int_{\mathbb{R}^d}V(x)\left[g_{x,x}(u_n)-g_{x,x}(u)\right](u_n-u)dx\nonumber\\
& -\int_{\mathbb{R}^d}b(x)\left[\vert u_n\vert^{p(x)-2}u_n-\vert u\vert^{p(x)-2}u\right](u_n-u)dx\nonumber\\
& \geq \langle J^{'}_{s,G_{x,y}}(u_n)-J^{'}_{s,G_{x,y}}(u),u_n-u\rangle+4\int_{\mathbb{R}^d}V(x)\widehat{G}_x\left(\frac{u_n-u}{2}\right)dx\nonumber\\
&-\int_{\mathbb{R}^d}b(x)\left(\vert u_n\vert^{p(x)-1}+\vert u\vert^{p(x)-1}\right)(u_n-u)dx\nonumber\\
& \geq \langle J^{'}_{s,G_{x,y}}(u_n)-J^{'}_{s,G_{x,y}}(u),u_n-u\rangle+4V_0\int_{\mathbb{R}^d}\widehat{G}_x\left(\frac{u_n-u}{2}\right)dx\nonumber\\
&-\int_{\mathbb{R}^d}b(x)\left(\vert u_n\vert+\vert u\vert\right)^{p(x)-1}(u_n-u)dx\nonumber\\
& \geq \langle J^{'}_{s,G_{x,y}}(u_n)-J^{'}_{s,G_{x,y}}(u),u_n-u\rangle\nonumber\\
&+4V_0\min \left\lbrace \left\Vert\frac{u_n-u}{2}\right\Vert_{L^{\widehat{G}_x}(\Omega)}^{g^-},\left\Vert \frac{u_n-u}{2}\right\Vert_{L^{\widehat{G}_x}(\Omega)}^{g^+}\right\rbrace\nonumber\\
& -\Vert b\Vert_{\delta^{'}(\cdot)}\left\Vert \left(\vert u_n\vert+\vert u\vert\right)^{p(x)-1}\right\Vert_{\frac{\delta(\cdot) p(\cdot)}{p(\cdot)-1}}\Vert u_n-u\Vert_{\delta(\cdot) p(\cdot)}\nonumber\\
& \geq \langle J^{'}_{s,G_{x,y}}(u_n)-J^{'}_{s,G_{x,y}}(u),u_n-u\rangle\nonumber\\
&+4V_0\min \left\lbrace \left\Vert\frac{u_n-u}{2}\right\Vert_{L^{\widehat{G}_x}(\Omega)}^{g^-},\left\Vert \frac{u_n-u}{2}\right\Vert_{L^{\widehat{G}_x}(\Omega)}^{g^+}\right\rbrace\nonumber\\
& -\Vert b\Vert_{\delta^{'}(\cdot)}\left(\Vert  \vert u_n\vert +\vert u\vert\Vert^{p^--1}_{\delta(\cdot) p(\cdot)}+\Vert \vert u_n\vert +\vert u\vert\Vert^{p^+-1}_{\delta(\cdot) p(\cdot)}\right)\Vert u_n-u\Vert_{\delta(\cdot) p(\cdot)}.
  \end{align}
  Combining \eqref{2eq82} and \eqref{2eq83}, and using Theorem \ref{compact}, we get
  $$\limsup_{n\rightarrow +\infty}\langle J^{'}_{s,G_{x,y}}(u_n)-J^{'}_{s,G_{x,y}}(u),u_n-u\rangle\leq 0.$$
  It follows, by \eqref{2eq80} and proposition \ref{s+}, that
  $$u_n\longrightarrow u\  \text{in }\ \mathbb{E}.$$
  Thus, the proof.
  \end{proof}
  \begin{proof}[Proof of Theorem \ref{thmv}] Since $I\in C^{1}\left( \mathbb{E}, \mathbb{R}\right)$, from Lemmas \ref{lemv1} and \ref{lemps}, it yields that the global minimum of the functional $I$ is achieved on  $\mathbb{E}$. Namely, there exists $u\in \mathbb{E}$  such that
  $$c=I(u)=\min_{v\in \mathbb{E}} I(v).$$
  It follows that $u$ is a critical point of $I$, that is, $I^{'}(u)=0$. Hence, $u\in \mathbb{E}$ is a weak solution for problem \eqref{PV}. To the end of the proof it remains to show that $u\not\equiv 0$. In this way, let $v\in \mathbb{E}\setminus\lbrace 0\rbrace$ and $t>0$. By Lemmas \ref{lem2} and \ref{lem33}, we obtain
  \begin{align*}
      c&\leq I(tv)=J_{G_{x,y}}(tv)+\int_{\mathbb{R}^d}V(x)\widehat{G}_x(tv)dx-\int_{\mathbb{R}^d}\frac{1}{p(x)} b(x)\vert tv\vert^{p(x)}dx\\
      & \leq \max\left\lbrace \Vert tv \Vert_{\mathbb{E}}^{g^-},\Vert tv \Vert_{\mathbb{E}}^{g^+} \right\rbrace-\frac{\min\lbrace t^{p^-},t^{p^+}\rbrace}{p^+}\int_{\mathbb{R}^d} b(x)\vert v\vert^{p(x)}dx.
  \end{align*}
  Therefore, since $p^-\leq p^+<g^-\leq g^+$,
  $$c\leq I(tv)<0,\ \ \text{for}\ t \ \text{small enough}.$$
  Thus, $u\in \mathbb{E}\setminus\lbrace 0\rbrace$. This ends the proof.
  \end{proof}
  \section{Final comments}
  In this section, we give some particular cases of the general fractional
  Musielak-Sobolev space. Then,  we present some interesting
  open questions.
  \subsection{Some examples}
  The novelty of this work is that our theorems are valid for a large class of Sobolev spaces and equations. To illustrate the degree of
the generality of our results, let us consider some cases depending
on the generalized N-function $G_{x,y}$ that are covered in this
article.
  \begin{enumerate}
      \item[$(1)$] Let
      $G_{x,y}(t)=\frac{1}{p(x,y)}\vert t\vert^{p(x,y)}$, for all $(x,y)\in \Omega\times\Omega$ and $t\in \mathbb{R}$, where $p:\Omega\times\Omega\longrightarrow (1,+\infty)$ is a continuous function satisfying
      $$1<p^-\leq p(x,y)\leq p^+,\ \text{for all}\ (x,y)\in \Omega\times\Omega.$$
      In this case the function $G_{x,y}$ satisfies the assumptions $(g_1)-(g_5)$ and \eqref{bf}. Moreover,
the Musielak-Sobolev space $W^{s,G_{x,y}}(\Omega)$ becomes the fractional Sobolev space with variable exponent $W^{s,p(\cdot,\cdot)}(\mathbb{R}^d)$ and the fractional  Musielak $g_{x,y}$-Laplace operator turns into the fractional $p(x,y)$-Laplacian. Therefore, our results (Theorems \ref{thm1}, \ref{comp1}, \ref{compact}, \ref{lions}, \ref{thms}, and \ref{thmv}) remain valid for fractional Sobolev space with variable exponent which are related to the main results shown in \cite{M33,M60,M96}.
\item[$(2)$] Let $G_{x,y}(t)=M(t)$, for all $(x,y)\in \Omega\times\Omega$ and $t\in \mathbb{R}$, where $\displaystyle{M(t):=\int_{0}^{\vert t\vert}m(\tau)d\tau}$ is an N-function (for  definition see \cite{M34}) satisfying the following conditions
\begin{equation*}\label{3eqg}
 1<m^-\leq \frac{m(t)t}{M(t)}\leq m^+<m^-_*:=\frac{dm^-}{d-sm^-}<+\infty,\ \text{for all}\ t>0
\end{equation*}
and
\begin{equation*}\label{g55}
  \int_{0}^{1} \frac{M^{-1}(\tau)}{\tau^{\frac{d+s}{d}}} d\tau<+\infty\ \ \text{and}\ \ \int_{1}^{+\infty} \frac{M^{-1}(\tau)}{\tau^{\frac{d+s}{d}}} d\tau=+\infty.
\end{equation*}
 It is clear that the generalized N-function $G_{x,y}$ satisfies the assumptions $(g_1)-(g_5)$ and \eqref{bf}. This implies that the fractional  Musielak-Sobolev space $W^{s,G_{x,y}}(\Omega)$ recover the fractional Orlicz-Sobolev space $W^{s,M}(\Omega)$ which firstly introduced in \cite{M28}. Thus, our principal Theorems \ref{thm1}, \ref{comp1}, \ref{compact}, \ref{lions}, \ref{thms}, and \ref{thmv} extend the results  obtained in \cite{M47,M13,M97} for the fractional Orlicz-Sobolev space  $W^{s,M}(\Omega)$.
 \item[$(3)$]  Let $G_{x,y}(t)=\frac{1}{p}\vert t\vert^p+\frac{1}{q}b(x,y) \vert t\vert^q $, for all $(x,y)\in \Omega\times\Omega$ and $t\in \mathbb{R}$, where  $b\in L^{\infty}(\Omega\times\Omega)$ is a non-negative symmetric function and $1<p<q<d$. Thus, we can rewrite the fractional Musielak $g_{x,y}$-Laplace operator as follows:
 \begin{equation}\label{2eq100}
(-\Delta)^{s}_{g_{x,y}}u:=(-\Delta)^s_pu+(-\Delta )^s_{b,q}u,
\end{equation}
where (up to multiplicative constant) $(-\Delta)^s_p$ is the so-called fractional p-Laplacian operator and
$(-\Delta )^s_{b,q}$ is the anisotropic fractional p-Laplacian defined as
$$(-\Delta )^s_{b,q}u(x):=\text{p.v.}\int_{\mathbb{R}^d}b(x,y)\frac{\vert u(x)-u(u)\vert^{q-2}(u(x)-u(y))}{\vert x-y\vert^{d+sq}}\frac{dy}{\vert x-y\vert^{d+s}},\ \ \text{for all}\ x\in \mathbb{R}^d
$$
 where p.v. is a commonly used abbreviation for "in the principle value sense".
 We would like to mention that problem \eqref{PV} with the operator \eqref{2eq100} is called nonlocal double phase problem, see \cite{M5}.
  \end{enumerate}
  \subsection{Perspectives and open problems}
  We summarize some open problems which are deduced from our work as follows:
  \begin{enumerate}
     \item[$(1)$] We would like to mention that Theorem \ref{thm1} is not optimal. It is worth noting that  the authors in  \cite{Cianchi5} proved  the optimal continuous
      embedding theorems for the fractional Orlicz-Sobolev spaces.
     Hence, it is a natural question to  see if the optimal embedding theorems obtained in \cite{Cianchi5}  can be extended to the fractional  Musielak-Sobolev spaces.
     \item[$(2)$] The $\Delta_2$-condition ( $(g_4)$ ) and assumption \eqref{bf}  played a key role in the proof of the continuous embedding theorem in $\mathbb{R}^{d}$ (Theorem
     \ref{thm1}).  Note that, Theorem \ref{thm1} is the basic tool in proving  Theorems \ref{comp1},
\ref{compact}, \ref{lions} and \ref{thms}. We do not have any
knowledge about the proof of Theorem \ref{thm1} without
$\Delta_2$-condition ( $(g_4)$ ) and assumption \eqref{bf}.
  \end{enumerate}

\end{document}